\documentclass{article}
\usepackage{todonotes}
\usepackage{authblk}
\usepackage{comment}

\usepackage{textgreek}	
\usepackage{imakeidx}	
\usepackage{multicol}
\usepackage[utf8]{inputenc}
\usepackage{changepage}
\usepackage{units}
\usepackage{amsfonts,amsmath,amssymb,advdate,mathrsfs,xfrac,dsfont,bm,stmaryrd,bm,amsthm}
\SetSymbolFont{stmry}{bold}{U}{stmry}{m}{n}
\usepackage{caption,tabularx,systeme,xcolor}

\usepackage{algorithm2e}
\usepackage{comment}
\usepackage{hyperref}
\usepackage{mathtools}
\usepackage{dirtytalk}
\usepackage{bbm}
\usepackage{bm}
\usepackage{mathtools,cases}

\usepackage{newpxtext}
\usepackage{microtype}
\usepackage[most]{tcolorbox}
\usepackage[font=small,labelfont=bf,
   justification=justified,
   format=plain]{caption} 
\usepackage{minitoc}

\usepackage{cancel}
\usepackage{xcolor} 
\usepackage{comment}
\usepackage{subfig}
\usepackage{lipsum}
\usepackage{graphicx} 
\usepackage{subcaption}
\usepackage{booktabs}

\newcommand{\eps}{\varepsilon}

\usepackage{svg}

\usepackage{marginnote}
\usepackage{setspace}
\usepackage{comment}

\usepackage{ulem}
\newtheorem{theorem}{Theorem}[section]

\newtheorem{proposition}{Proposition}[section]

\theoremstyle{remark}
\newtheorem{remark}{Remark}[section]

\newcommand{\corr}[1]{{\color{black}{{}#1}}}
\newcommand{\corrThm}[1]{{\color{black}{{}#1}}}

\title{Some mathematical models for flagellar activation mechanisms}

\author[1]{François Alouges}
\affil[1]{Centre Borelli, ENS Paris-Saclay, CNRS\\ Université Paris-Saclay, 4 avenue des sciences\\
91190 Gif-sur-Yvette, France\\ Institut Universitaire de France (IUF)}

\author[2]{Irene Anello}
\author[2]{Antonio DeSimone}
\affil[2]{Scuola Internazionale Superiore di Studi Avanzati\\ via Bonomea 265, I-34136 Trieste, Italy}

\author[3]{Aline Lefebvre-Lepot}
\affil[3]{CNRS, Fédération de Mathématiques de CentraleSupélec\\
9, rue Joliot Curie,
91190 Gif-sur-Yvette}

\author[4]{Jessie Levillain}
\affil[4]{CMAP, CNRS, École polytechnique\\ Institut Polytechnique de Paris \\
route de Saclay, 91120 Palaiseau, France}

\begin{document}

\maketitle

\begin{abstract}
\corr{This paper focuses on studying a model for dyneins,  cytoskeletal motor proteins responsible for axonemal activity}. The model is a coupled system of partial differential equations inspired by \cite{Julicher1995,julicher_molecular_1998} and incorporating two rows of molecular motors between microtubules filaments. Existence and uniqueness of a solution is proved, together with the presence of a supercritical Hopf bifurcation. Additionally, numerical simulations are provided to illustrate the theoretical results. A brief study on the generalization to $N$-rows is also included.
\end{abstract}

\section*{Introduction}
\corr{
Molecular motors are proteins that move along tubulin filaments, converting chemical energy (ATP) into mechanical work. Our focus is on dyneins, a specific type of molecular motor found in the axoneme, the cytoskeletal structure of cilia. The axoneme consists of nine pairs of tubulin filaments\footnote{On some biological systems there is also a central pair of microtubules \cite{cells8020160}. We don't consider it in the present study.}, known as microtubules, and functions as an active structure due to the dyneins positioned between adjacent microtubule pairs. These molecular motors attach to one microtubule pair and walk along the neighboring pair, generating local sliding between them. When the microtubules are considered elastic, as in cilia, this sliding induces bending. Along the microtubules, the molecular structure is periodic with periodicity $\ell \ll L$ where $L$ is the total length of the filament.

\begin{figure}[!ht]
\centering
\includegraphics[width=0.6\linewidth]{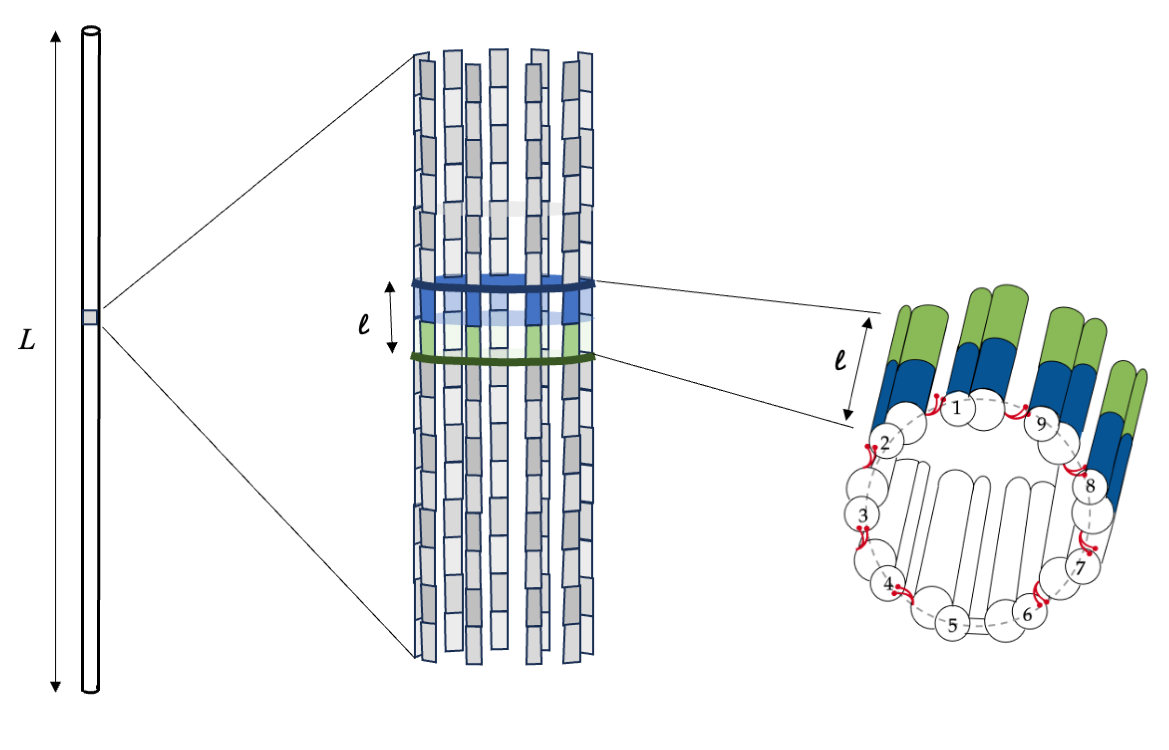}
  \caption {\corr{A microscopic slice of the axoneme (on the right), zoomed in from the whole filament (on the left). The periodic structure is highlighted by an alternation of green and blue sections, modeling preferred binding sites for molecular motors.}}
        \label{fig:axoneme}
\end{figure}
    
We analyze a microscopic slice of thickness $\ell$ of the microtubules (See Figure \ref{fig:axoneme}), 
focusing exclusively on the interaction between molecular motors and the filaments. At this scale, the axoneme is considered {\textit{isolated}}, i.e. there is no fluid to be taken into account, and the filaments are treated as rigid. The ultimate objective of this article is to model the entire structure of such a slice of the isolated axoneme, composed of nine microtubule pairs. 

As a starting point, we adopt the \textit{one-row model} described in \cite{Julicher1995, Julicher1997Modeling, julicher_molecular_1998,Julicher1997, guerin2011dynamical}, which characterizes the collective motion of a single row of molecular motors attached to a moving filament against a fixed counterpart. A key finding of this model is the emergence of mesoscopic order from the independent activity of individual molecular motors. When an elastic force is present, an analysis of the linearized system coming from the model shows that the system might possess a supercritical Hopf bifurcation instability \cite{guerin2011dynamical}. Numerical simulations \cite{guerin2011dynamical,Julicher1995, Julicher1997Modeling,Julicher1997, julicher_molecular_1998} confirm this kind of behavior.  

Building on the one-row model, and in order to model the circular structure of the slice of the axoneme, we introduce a second row of molecular motors, as shown in Figure~\ref{fig:2tubules_cross}(b), forming a \textit{two-row model} that symmetrizes the previous framework. This simplified structure was already suggested e.g. in \cite{Sartori2016} Fig. 1. A \& B, or in  \cite{riedel-kruse_how_2007} Fig. 2., A \& B, though not studied as an isolated system. We further generalize this to an \textit{N-row model}, where $N$ represents the number of motor rows -- typically 8 -- and, consequently, the number of microtubule pairs. Both models are analytically described in the next section.  

   \begin{figure}[!ht]           
\centering
 \subfloat[]{\includegraphics[width = 0.25 \linewidth]{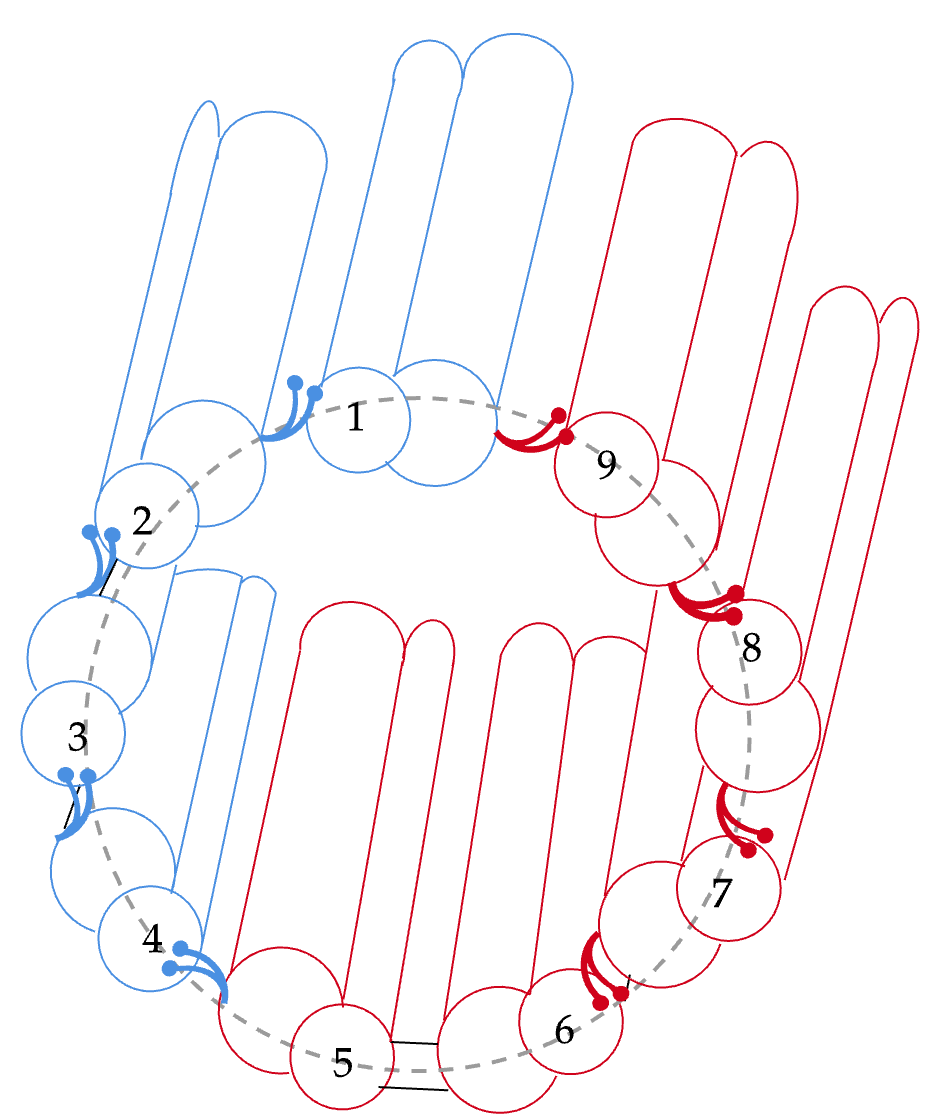}}
 \hfil
    \subfloat[]{\includegraphics[width = 0.25 \linewidth]{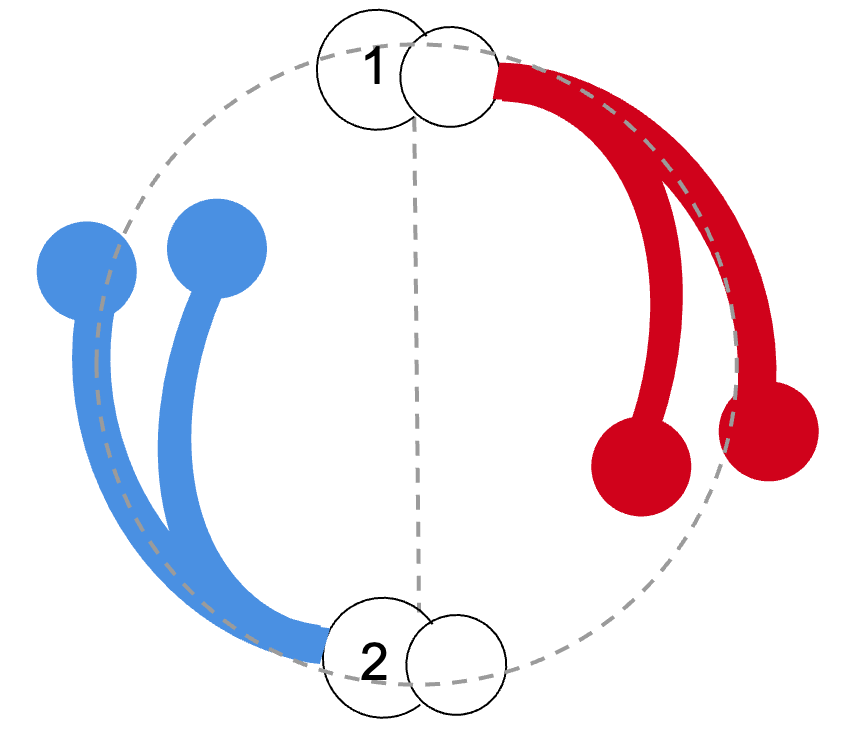}}
\hfil
    \subfloat[]{\includegraphics[width=0.5\linewidth]{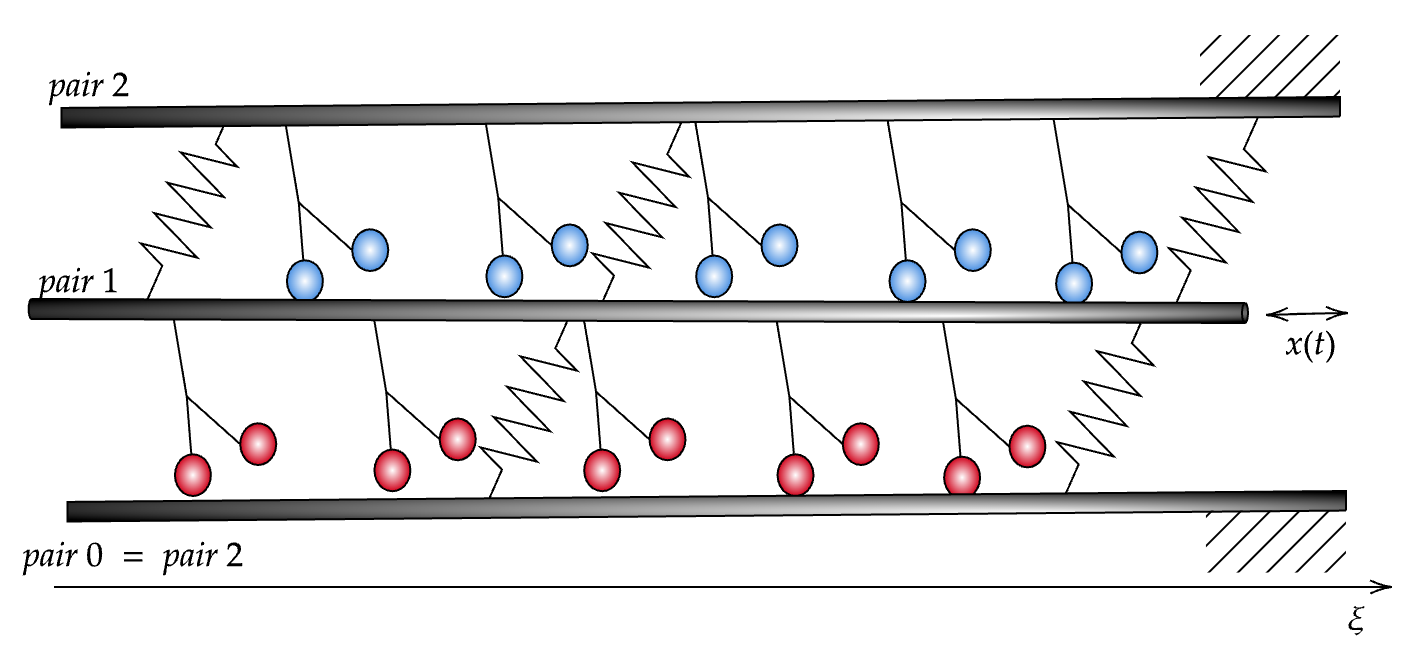}}

  \caption{(a) Cross section of the axoneme. (b) Cross section of the axoneme where we group the filaments 1-4 (blue) and 5-9 (red) together. (c) Unfolding of the axoneme in (b) with two opposite microtubule pairs.}
\label{fig:2tubules_cross}
    \end{figure}
    
The novelty of our work lies in four key contributions: (i) We generalize the one-row model by first introducing an alternative description for the isolated axoneme: the two-row model. We then use this model as a first step towards a more realistic $N$-row model.
(ii) Then we prove the existence and uniqueness of solutions for the two-row model which is composed of coupled non-linear PDEs. This result fills a gap in the existing theoretical literature and also extends to the one-row or $N$-row models. (iii) We establish rigorously the presence of a supercritical Hopf bifurcation in the system for general parameters. The non-linearity of the equations is treated through the Center Manifold technique. We further compute the amplitude of the periodic solution after the instability has occurred. (iv) Eventually, we introduce a specific upwind scheme to perform numerical simulations on the \textit{N-row model} for $ N \geq 1 $.}

\corr{The paper is organized as follows: the first section presents the derivation of the one-row, two-row and $N$-row models and the motivation underlying the increasing biology complexity. In the second section, we state and prove the existence and uniqueness of the solutions (Theorem \ref{th: Ex and Uniqueness}) and the existence of the Hopf bifurcation (Theorem \ref{th: Hopf bifurcation}). In the final section, we describe the numerical scheme employed in the paper, present the numerical simulations, and explore the questions arising from the $N$-row structure.}

\section{Motivation and Modeling}
\subsection{The one-row model}
\corr{In all the presented models, the molecular motors are described by a generalized version of the stochastic two-state model proposed in \cite{Julicher1995}. The key difference between the one-, two-, and $ N $-row models lies in the number of motor rows, which determines the number of filaments representing the slice of the isolated axoneme.

In the two-state model, each motor lies between two filaments. It is attached to one of them, and has two identical heads (head $A$ and head $B$), each of which can be either bound or unbound to the opposite filament. The motor can exist in two distinct chemical states: state $A$, where head $A$ is bound and head $B$ is unbound, and state $B$, where the opposite happens. Importantly, the model assumes that both heads cannot be simultaneously unbound or bound. In the one-row model, motors are attached to the moving top filament and exert a force on the stationary bottom filament. Between the pair of filaments, passive elastic and viscous elements resist the motion. They are modeled by the positive constants $k$ (elastic coefficient) and $\eta$ (viscous coefficient), respectively.

Along its length the filaments present a microscopic periodic structure of period $\ell$ that drives the chemical behavior of the heads. Therefore, we consider a slice of thickness $\ell$ of the isolated axoneme. Then, each state $ i = A, B $ experiences a ($\ell$-periodic) potential energy $ W_i(\xi) $ at position $ \xi \in [0,\ell] $ on the fixed bottom filament. The $\ell$-periodic transition rates $\omega_A(\xi)$ and $\omega_B(\xi)$ represent the probability per unit time for a motor to switch from state $ A $ to state $ B $ and from state $B$ to state $A$ respectively. The transition rates also depend on the ATP concentration $\Omega$, i.e. $ \omega_i = \omega_i(\xi; \Omega) $, as ATP provides the energy for state transitions. A detailed explanation of this dependence is given in \cite{Julicher1997Modeling}. Note that more complex approaches exist in the literature, for which transition rates are also force-dependent. We refer the reader to \cite{Lattanzi2002, Lattanzi2004, riedel-kruse_how_2007} for detailed explanations on such models.

In the limit of an infinite number of motors on $[0,\ell]$, we introduce the densities $ \tilde{P}_{i}(\xi) \geq 0 $ (for $ i = A, B $), which represent the probability of finding a motor at position $ \xi $ in state $ i $. These densities satisfy $ \tilde{P}_A(\xi) + \tilde{P}_B(\xi) = 1/\ell $, reflecting the fact that a motor is always in either state $A$ or state $B$. The position of the top filament is denoted by $ x(t) $, with velocity $ v(t) = \frac{d}{dt}x(t) $.

The one-row model, as described in \cite{Julicher1995}, expresses the motion of the system through a transport equation for $ P = \tilde{P}_A $:
\begin{equation} \label{eq: full pde 1row}
\left\{\begin{array}{l}
     \partial_t P(\xi,t) + v(t)\partial_{\xi} P(\xi,t) = -( \omega_A +\omega_B) P(\xi,t) + \omega_B/\ell  \\
     v(t) = \displaystyle \frac{1}{\eta}  \left(\int\limits_0^{\ell}  P(\xi,t)\partial_{\xi} \Delta W(\xi)d{\xi} - kx(t)\right), 
\end{array}\right.   
\end{equation} 
where $ \Delta W(\xi) = W_B(\xi) - W_A(\xi) $. The second equation in system \eqref{eq: full pde 1row} represents the force balance on the moving filament:
\begin{equation*}
    f_{ext} - \eta v - kx + f_{mot} = 0,
\end{equation*}
with zero external force ($ f_{ext} = 0 $) since we don't consider fluid-structure interaction here and where the force due to the motors is given by
\begin{equation*}
f_{mot}(t) = \int\limits_0^{\ell} P(\xi,t)\partial_{\xi} \Delta W (\xi)d\xi.
\end{equation*}

The system \eqref{eq: full pde 1row} admits a stationary, non-moving solution:
\begin{equation*}
    P_{eq}(\xi; \Omega) = \frac{\omega_B}{\ell(\omega_A+\omega_B)},\quad v_{eq} = 0,\quad x_{eq} = \frac{1}{k}\int_0^l \,P_{eq}(\xi; \Omega)\partial_{\xi}\Delta W(\xi)d\xi.
\end{equation*}

Depending on the specific forms of the transition rates $ \omega_i $ and potentials $ W_i $, a threshold ATP concentration $ \Omega = \Omega_0 $ may exist, above which this stationary solution becomes unstable, leading to oscillatory behavior in the system (see \cite{guerin2011dynamical,Julicher1995, Julicher1997Modeling,Julicher1997, julicher_molecular_1998})}.

\subsection{The two-row model}\label{sec: N=2}
We now derive in detail the two-row model, which is depicted in Figure \ref{fig:2tubules_cross}(b). The two-row model is a simplified version of the circular structure of the axoneme  (\cite{Sartori2016,gadelha2023}), forming a cylinder with only two pairs of microtubules linked by two rows of motors, as shown in Figure \ref{fig:2tubules_cross}(a). This model is a first step towards the $N$-row model presented in section \ref{sec: Nlayers}, where $N$ corresponds to the number of motor rows. 

Moreover, the two-row model symmetrizes the one-row model.
In the one-row model \cite{Julicher1995}, even without any external force, the moving filaments experiences a non-zero equilibrium displacement $x(t)=x_{eq}$, which disrupts axonemal symmetry and causes bending at rest. Jülicher and Camalet mentioned this issue in \cite{camalet2000generic}, solving it by choosing symmetrical potentials, as in our case, and symmetrical transition rates, namely $|\omega_i(\xi+l/2)|=|\omega_i(-\xi+l/2)|$, with $\xi \in [0,\ell]$ and $i = A, B$. Instead, in this work, we consider general non-symmetrical transition rates and introduce a second row of molecular motors, producing a zero equilibrium $x_{eq}=0$ at rest.

 Unfolding the two-row structure, we may consider three pairs of filaments, numbered from 0 to 2 where the pairs 0 and 2 are identified and fixed, while the central pair 1, may move (See Figure \ref{fig:2tubules_cross}(c)). We call $x(t)$ its displacement and $v(t)= \frac{d}{dt}x(t)$ its velocity. We define $P_1(\xi,t) = \tilde{P}_{1,A}(\xi,t)$ and $P_2(\xi,t)= \tilde{P}_{2,A}(\xi,t)$ as the probabilities for the motors to be in state $A$ at position $\xi$ and time $t$, when attached to the pair 1 (bottom row) or to the pair 2 (top row), respectively. We assume that the sum of the transition rates is uniform, (see \cite{camalet2000generic,guerin2011dynamical}),
\begin{equation} \label{eq: uniformity om1+om2}
     \omega_A(\xi; \Omega) +\omega_B(\xi; \Omega) = a_0(\Omega),
\end{equation}
where $a_0(\Omega)>0$, and obtain two transport equations for $\xi \in [0, \ell]$ and $t>0$:
\begin{align} \label{eq: original full two layers P1 P2}
\left\{\begin{array}{ll}
    &\displaystyle \frac{\partial P_1}{\partial t}(\xi,t) + v(t) \frac{\partial P_1}{\partial \xi}(\xi,t) = -a_0(\Omega) P_1(\xi,t) + \frac{\omega_B(\xi; \Omega)}{\ell}, \\
   & \displaystyle \frac{\partial P_2}{\partial t}(\xi,t)  = -a_0(\Omega)P_2(\xi,t) + \frac{\omega_B(\xi-x(t); \Omega)}{\ell}.
    \end{array}
    \right.
\end{align}

Additionally, the force balance equation on the central filament reads:
\begin{equation}\label{eq: force balance}
    -2\eta v - 2kx +f_{mot} = 0,
\end{equation}
where as before, no external force is considered, $k x$ and $\eta v$ represent the elastic and viscous resistances, respectively and where $f_{mot}$ is the active force exerted by the motors, which is now given by
    \begin{equation}\label{eq: motor force two layers system}
    f_{mot}(t) = \int\limits_0^{\ell} (P_1(\xi,t)\partial_{\xi} \Delta W (\xi) - P_2(\xi,t)\partial_{\xi} \Delta W (\xi - x(t)))d\xi.
\end{equation}

Let $P(\xi, t) = P_1(\xi,t)$ and $Q(\xi,t) = P_2(\xi + x(t),t)$. Then, \eqref{eq: original full two layers P1 P2} becomes 
\begin{align}\label{eq: 2 layers P Q}
\left\{\begin{array}{ll}
    &\displaystyle   \frac{\partial P}{\partial t}(\xi,t) + v(t) \frac{\partial P}{\partial \xi}(\xi,t)  = -a_0(\Omega) P(\xi,t) + \frac{\omega_B(\xi)}{\ell},\\
   &\displaystyle \frac{\partial Q}{\partial t}(\xi,t) - v(t) \frac{\partial Q}{\partial \xi}(\xi,t)  = -a_0(\Omega) Q(\xi,t) + \frac{\omega_B(\xi)}{\ell}.
    \end{array}
    \right.
\end{align}
The periodicity of $P_1$ and $P_2$, enables us to rewrite the motor force as
\begin{equation*}
     f_{mot}(t) = \int\limits_0^{\ell} (P(\xi,t)- Q(\xi,t))\partial_{\xi} \Delta W (\xi)d\xi \,,
\end{equation*}
and obtain, from \eqref{eq: force balance},
\begin{equation}\label{eq: 2 layers v}
    v(t) = \dot x(t) = \displaystyle \frac{1}{2\eta} \left( \int\limits_0^{\ell} (P(\xi,t)- Q(\xi,t))\partial_{\xi} \Delta W (\xi)d\xi - 2kx(t) \right).
\end{equation}

Finally, we obtain the system of equations 

\begin{align} \label{eq: 2 layers full PDE}
    \left\{\begin{array}{lll}
    &\displaystyle   \frac{\partial P}{\partial t}(\xi,t) + v(t) \frac{\partial P}{\partial \xi}(\xi, t)  = -( \omega_A(\xi) +\omega_B(\xi)) P(\xi, t) + \frac{\omega_B(\xi)}{\ell},\\[10pt]
   &\displaystyle \frac{\partial Q}{\partial t}(\xi,t) - v(t) \frac{\partial Q}{\partial \xi}(\xi,t)  = -( \omega_A(\xi) +\omega_B(\xi)) Q(\xi,t) + \frac{\omega_B(\xi)}{\ell}, \\
    &   v(t) = \displaystyle \frac{1}{2\eta} \left( \int\limits_0^{\ell} (P(\xi,t)- Q(\xi,t))\partial_{\xi} \Delta W (\xi)d\xi - 2kx(t) \right),\\[10pt]
    & P(0,t) = P(\ell,t), \quad Q(0,t) = Q(\ell,t).
    \end{array}
    \right.
\end{align}
The equilibrium state for this system is, 
\begin{equation*}
     \left\{
    \begin{array}{l}
      v_{eq} = 0,\\
      x_{eq} = 0,\\
      P_{eq} = Q_{eq} = \displaystyle \frac{\omega_B}{\ell (\omega_A + \omega_B)}.
    \end{array}
  \right.
\end{equation*}
Notice that while the probability equilibrium is the same as in the one-row model, the equilibrium position is now zero.
This system will then be studied both theoretically and numerically in the next sections, assuming suitable initial conditions for $P$ and $Q$.

\corr{\subsection{The $N$-row model}\label{sec: Nlayers}}

Starting from two rows of molecular motors, we extend the model from section \ref{sec: N=2} to $N \geq 2$ rows of molecular motors. \\
As illustrated in figure \ref{fig:simple_axoneme} The system is then composed of $N+1$ microtubule doublets on the outside, arranged in a circle, and no central pair. All filaments have the same polarity, meaning that between two microtubule pairs, motors move towards the base. 
\begin{figure}[!ht]
\centering
\includegraphics[width=0.5\linewidth]{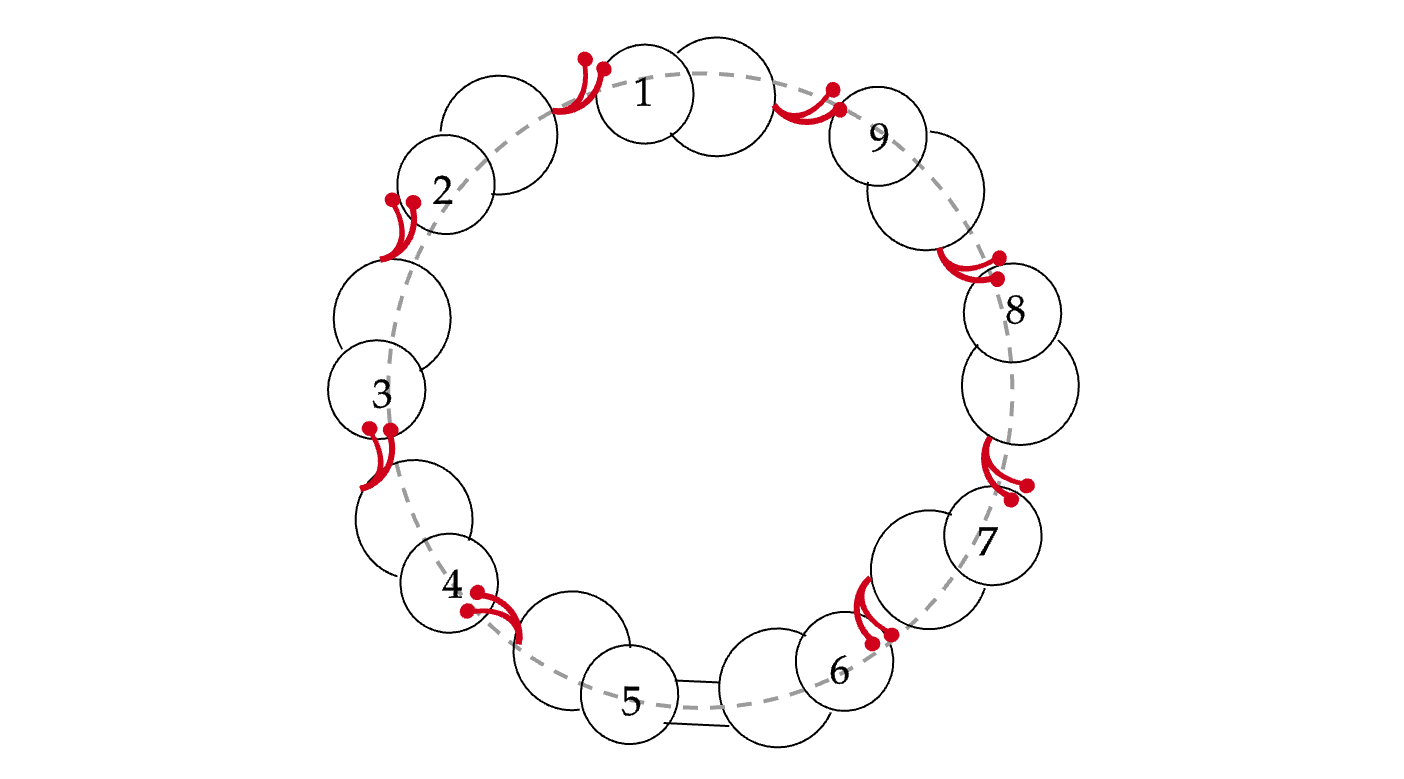}

  \caption {Simplified axoneme with $9$ microtubule doublets and no central pair \cite{Yagi1995}.}
        \label{fig:simple_axoneme}
    \end{figure}
As before, filament pairs $0$ and $N$ are considered to be the same pair in order to take into account the cylindrical periodic structure. They have no shift between each other, as shown in the split view Figure \ref{fig:Ntubshift}.
Moreover, all microtubule pairs are assumed inextensible, rigid and at a constant distance from each other. They are held together by elastic and viscous elements that do not permit infinite sliding between them.
We look at molecular motors in a periodicity cell of an $N$-axoneme, of length $\ell$.
\begin{figure}[ht!]
    \centering
    \includegraphics[width= 0.7 \textwidth]{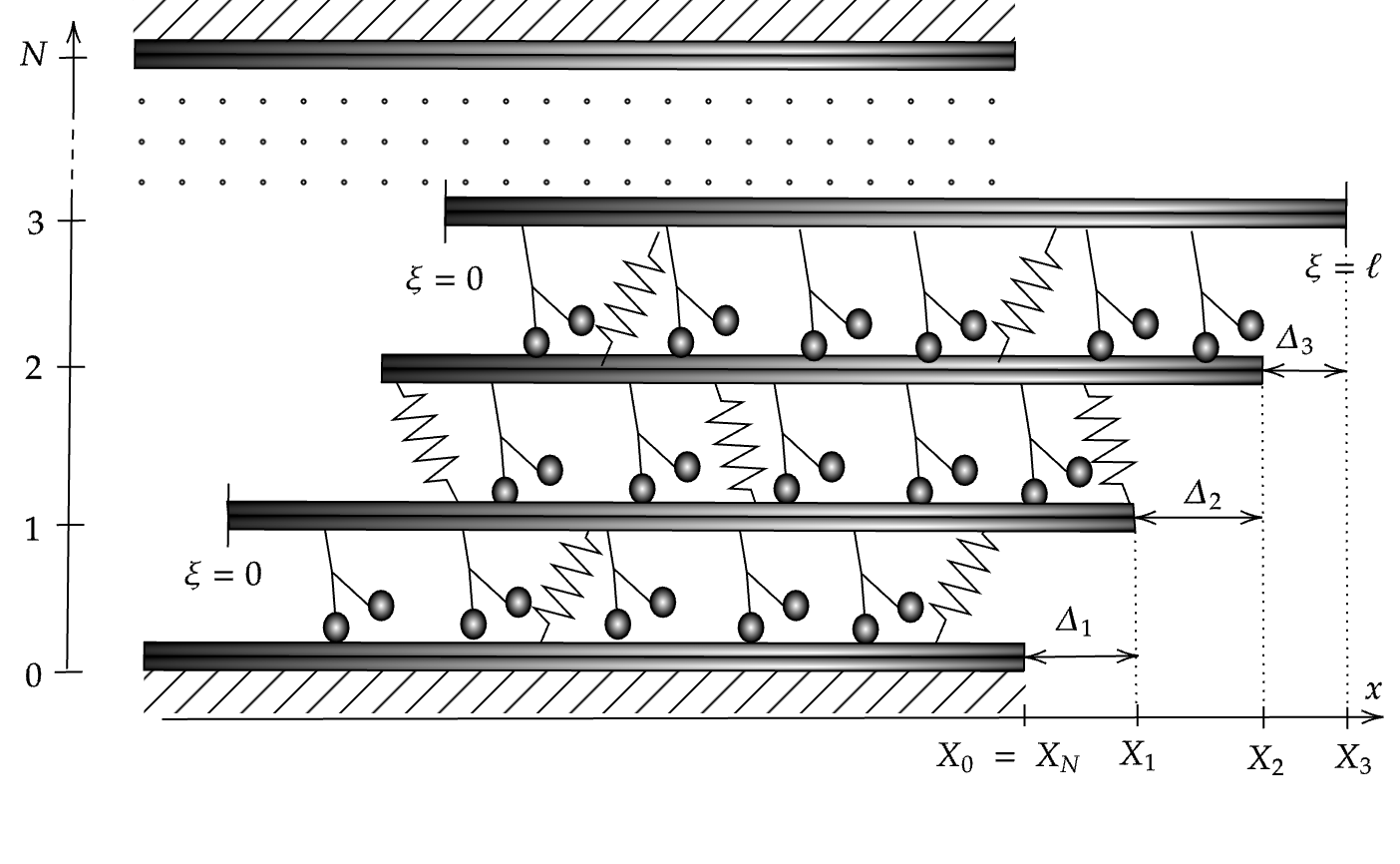}
    \caption{Motors and tubule shift in the $N$-row model (here the horizontal axis $x$ shows absolute tubule displacement instead of motor position $\xi$ in a periodicity cell).}
    \label{fig:Ntubshift}
\end{figure}

For $i \in \{ 0, \dots, N \}$, we denote by $x_i$ the horizontal shift of the $i$-th tubule, as shown in Figure \ref{fig:Ntubshift}.
We measure displacement with respect to the $0$-th filament: $ x_0 = x_N = 0$. For $1 \leq i \leq N$, we introduce $\Delta_i = 
x_i - x_{i-1}$, the relative displacement between filaments $i$ and $i-1$. It follows that $\sum_{i=1}^N\Delta_i(t) =0$. The shifting speed of the $i$-th filament is defined by $v_i = \dot{x}_i - \dot{x}_{i-1} = \dot{\Delta}_i$, and we have
\begin{equation}\label{modA: dotDelta=0}
   \sum\limits_{i=1}^N v_i = \sum_{i=1}^N\dot{\Delta}_i(t)=0.
\end{equation}

Once again, the variable $\xi \in [0, \ell]$ represents the local variable along one pair, as in the two-row model. We denote by $P_i(\xi, t) =\tilde{P}_{i,A}(\xi, t)  $ the density of molecular motors at position $\xi$ and time $t$ attached to the $i$-th filament and who are in state $A$, walking on the $(i-1)$-th filament. 
Following again section \ref{sec: N=2}, for $j \in \{ 1, \dots, N \}$, $t>0$ and $\xi \in [0, \ell]$, we define $Q_j$ as $Q_j(\xi,t) = P_j(\xi + \Delta_j(t),t)$
and we finally obtain the system at any $t>0$ and  $\xi \in [0, \ell]$
\begin{equation}
\left \{
  \begin{aligned}
\partial_t Q_j + v_j \partial_{\xi} Q_j = -(\omega_A + \omega_B) Q_j + \omega_B/\ell, && \text{for $j \in \{ 1, \dots, N \}$,}\\
   \eta (v_ i-v_{i+1}) = \displaystyle\int\limits_0^{\ell} (Q_i - Q_{i+1}) \partial_{\xi}\Delta W - k(\Delta_i - \Delta_{i+1}), && \text{for $i \in \{ 1, \dots, N-1 \}$},\\
   \sum_{i=1}^N\dot{\Delta}_i(t)=0  \text{ and } \dot x_N = 0,
  \end{aligned} \right.
  \label{eq: Nlayers} 
\end{equation}
where the first $N$ equations are transport equations for the motors probability density and the $N-1$ following equations are the force balances on the $i$-th filament with the motors and springs around it, for $i=1\ldots N-1$.

Like in previous models, the equilibrium state is, for all $1 \leq i \leq N-1$ and $1 \leq j \leq N$
\begin{equation*}
     \left\{
    \begin{array}{l}
      v_{eq,i} = 0,\\
      x_{eq,i} = 0,\\
      Q_{eq,j} = \displaystyle \frac{\omega_B}{\ell (\omega_A + \omega_B)}.
    \end{array}
  \right.
\end{equation*}
Notice that again, the probability equilibrium is the same as in the previous models and, as in the two-row setting, the equilibrium position is zero.

\section{Theoretical results}
\corr{In this section we state and prove two theorems, giving the existence and uniqueness of solutions  and existence of a Hopf bifurcation for two-row model \eqref{eq: 2 layers full PDE}. We then give some theoretical consideration on the general $N$-row model \eqref{eq: Nlayers}.}

Let us define $C_{\#}^1([0,\ell])$ as the set of restrictions to $[0,\ell]$ of $C^1$ and $\ell$-periodic functions over $\mathbb{R}$, then the first theorem reads as follows:
\begin{theorem}[Existence and uniqueness] \label{th: Ex and Uniqueness}
Let us fix $\ell>0$. Assume $\omega_A(\xi;\Omega)$ and $\omega_B(\xi;\Omega)$ as in \eqref{eq: uniformity om1+om2}. Moreover, assume $\omega_B$ and $\Delta W$ to be at least $C_{\#}^1([0,\ell])$.
If the initial data $P(\xi, 0)$ and $Q(\xi, 0)$ are $C_{\#}^1([0,\ell])$ and $x(0)=0$, then the system of equations \eqref{eq: 2 layers full PDE}, admits a unique solution $P, \, Q \in C^1([0,\ell]\times \mathbb{R}_+)$, with $\xi \mapsto P(\xi,\cdot)$ and $\xi \mapsto Q(\xi,\cdot)$ in $C_{\#}^1([0,\ell])$, and $x \in C^1(\mathbb{R}_+)$.
\end{theorem}
For the second result we chose $\Delta W (\xi) = U \cos \left(\displaystyle \frac{2 \pi \xi }{\ell}\right)$, as in \cite{camalet2000generic}. Moreover, since the two motor heads are identical, then $\omega_B(\xi) = \omega_A(\xi + \ell/2)$. Taking the transition rates' periodicity into account, we use the following Fourier expansion
\begin{equation} \label{eq: om2, om1 general choices}
\omega_B(\xi; \Omega) = \displaystyle \frac{a_0(\Omega)}{2} + \sum\limits_{n = 2k + 1,\, k \geq 0} \left( a_n(\Omega) \cos{\frac{2n \pi \xi}{\ell}} + b_n(\Omega)\sin{\frac{2n \pi \xi}{\ell}}\right),
\end{equation}
with $a_n(\Omega)$ and $b_n(\Omega)$ real coefficients. The even coefficient disappear due to the fact that $\omega_B(\xi) = \omega_A(\xi + \ell/2)$ and $\omega_A(\xi) + \omega_B(\xi) = a_0$.
\begin{theorem}[Hopf bifurcation]\label{th: Hopf bifurcation} \corrThm{In addition to the hypothesis above, assume that $\Omega \geq 0$, $a_0(\Omega)>0$, $a_1(\Omega) \neq 0$ and that $a_i(\Omega)$ with $i=0,1$ are at least $C^1(\mathbb{R}_+)$. Furthermore, let us define 
\begin{equation}    \label{eq: tau}
    \tau (\Omega) := -\frac{1}{4}\left(2a_0(\Omega)+ \frac{\zeta  \ell}{\pi} +  2 \lambda \frac{a_1(\Omega)}{a_0(\Omega) \ell}\right),
\end{equation}
with $\zeta = 2 \pi k / \eta \ell$ and $\lambda = 2 \pi^2 U / \eta \ell$.}
Suppose there exists $\Omega_0 \in \mathbb{R}_+$ such that $\tau(\Omega_0)=0$ and $\tau'(\Omega_0)>0$, then the solutions $P(\xi,t;\Omega)$, $Q(\xi,t;\Omega)$ and $x(t;\Omega)$ of the system \eqref{eq: 2 layers full PDE} show a super-critical Hopf bifurcation in time near the bifurcation value $\Omega_0$ and near the fixed point 
    \begin{equation*}
        \left(P_{eq}(\xi;\Omega),Q_{eq}(\xi;\Omega),x_{eq}(\Omega)\right) = \left(\frac{\omega_B(\xi;\Omega)}{a_0(\Omega)\ell},\frac{\omega_B(\xi;\Omega)}{a_0(\Omega)\ell},0\right).
    \end{equation*}
    In particular, the fixed point is asymptotically stable for $\Omega < \Omega_0$, and unstable for $\Omega > \Omega_0$. Moreover there exists an asymptotically stable periodic orbit of frequency $\omega_0$ given by 
    \corrThm{and frequency
\begin{equation} \label{th:freq}
\omega_0 = \sqrt{\frac{\zeta  \ell a_0(\Omega_0)}{2 \pi}}
\end{equation}}
for $\Omega > \Omega_0$. In that case, the displacement $x(t)$ has an amplitude given by 
\begin{equation} \label{eq: th amplitude Hopf}
\rho = \sqrt{-\frac{(\Omega-\Omega_0)\tau'(\Omega_0)}{\tilde{\tau}}} +o\left(\sqrt{\Omega-\Omega_0}\right),
\end{equation}
where $\displaystyle \tilde{\tau} = - \frac{3 \pi \zeta}{4 \ell} \left(\frac{\pi a_0(\Omega_0) + \ell \zeta}{\pi a_0(\Omega_0) + 2 \ell \zeta}\right)$. 
\end{theorem}

The methods introduced to prove the theorems can be also used to demonstrate similar results for the one-row model, and might prove useful to consider the general $N$-row structure. \corr{These two results rigorously prove the existence of a solution for system \eqref{eq: 2 layers full PDE} which bifurcates in time for a given value of the ATP concentration. The relative sliding $x(t; \Omega)$ between microtubules shows oscillatory motion when $\Omega>\Omega_0$, depending on the coefficients $a_0, \, a_1$ in \eqref{eq: om2, om1 general choices}. This happens as a result of the binding and unbinding coordinated activity of molecular motors, whose dynamics is represented by $P(\xi,t;\Omega)$ and $Q(\xi,t;\Omega)$. The behavior of the bifurcation will be investigated at the end of the following section in term on the way the coefficients $a_0, \, a_1$ depend on the ATP concentration $\Omega$. Notice that the other Fourier coefficients $a_n$, with $n>1$, and $b_n$, with $n\geq 1$, don't play any role in the onset of bifurcation but influence the probabilities' amplitude of oscillation.}

\subsection{Proof of Theorem \ref{th: Ex and Uniqueness} (existence and uniqueness)}

The proof follows a classical scheme in two steps: we first show local existence of the solution in time, and then extend it to $\mathbb{R}_+$.\\

\noindent \textbf{Step 1: local existence.}
Let $T>0$ be a time to be chosen afterwards and define the map $\psi: C^0([0,T]) \to C^0([0,T])$ as $\psi[x(\cdot)](t) = \displaystyle \int_0^t F[x(\cdot)](s)\,ds$, where 
\begin{equation}
    F[x(\cdot)](s) = \displaystyle \frac{1}{2\eta} \int\limits_0^{\ell} (P_x(\xi,s)- Q_x(\xi,s))\partial_{\xi} \Delta W (\xi)d\xi - \frac{k}{\eta} x(s),
    \label{eq: contraction}
\end{equation}
and the functions $P_x(\xi,t)$ and $Q_x(\xi,t)$ are defined as the solutions of \eqref{eq: 2 layers P Q} and $v(t) = \dot x(t)$. Notice that $P_x$ and $Q_x$ are explicitly given in their integral form by
\begin{equation}\label{eq: integral form for P}
P_{x}(\xi,t)= e^{-a_0 t}P (\xi - x(t)+ x(0), 0) +\displaystyle\frac{e^{-a_0 t}}{\ell}\displaystyle\int_0^t e^{a_0 s}\omega_B(\xi + x(s)-x(t))\,ds,
\end{equation}
and 
\begin{equation}\label{eq: integral form for Q}
Q_{x}(\xi,t)=  e^{-a_0 t}Q (\xi +x(t)-x(0), 0) +\displaystyle\frac{e^{-a_0 t}}{\ell}\displaystyle\int_0^t e^{a_0 s}\omega_B(\xi - x(s)+x(t))\,ds,
\end{equation}
which remain well-defined when $x$ is only in $C^0([0,T])$.

Notice that $(\tilde{x}, P_{\tilde{x}}, Q_{\tilde{x}})$ is a solution of the system (\ref{eq: 2 layers P Q}, \ref{eq: 2 layers v}) if and only if $\tilde{x}(t)$ is a fixed point of $\psi$, i.e. $\psi[\tilde{x}(\cdot)] = \tilde{x}(\cdot)$. In order to proceed, we prove that $\psi: C^0([0,T]) \to C^0([0,T])$ is a strict contraction for $T$ sufficiently small.\\

Let us take two functions $x_1$ and $x_2$ in $C^0([0,T])$, with $x_1(0)=x_2(0)=0$ and with the corresponding $P_{x_1}$, $Q_{x_1}$ and $P_{x_2}$, $Q_{x_2}$ defined 
by the integral formulations (\ref{eq: integral form for P}, \ref{eq: integral form for Q}). The initial conditions are identical $P_{x_1}(\xi, 0) = P_{x_2}(\xi, 0)=P(\xi)$ and $Q_{x_1}(\xi, 0) = Q_{x_2}(\xi, 0)=Q(\xi)$. We want to estimate the quantity 
\begin{equation}
    \psi[x_1(\cdot)](t)-\psi[x_2(\cdot)](t) = \int_0^t F[x_1(\cdot)](s)-F[x_2(\cdot)](s)\,ds.
\end{equation}
Using \eqref{eq: contraction}, we get
\begin{equation}\label{eq: F1-F2}
\begin{split}
        F[x_1(\cdot)](s)&-F[x_2(\cdot)](s) = \displaystyle \frac{1}{2\eta} \Bigg( \int\limits_0^{\ell} \Big(P_{x_1}(\xi, s) - P_{x_2}(\xi, s)  \\
        & - \big( Q_{x_1}(\xi, s) - Q_{x_2}(\xi, s) \big)\Big)\partial_{\xi} \Delta W (\xi)d\xi \Bigg)- \displaystyle\frac{k}{\eta} (x_1(s)-x_2(s)).
\end{split}
\end{equation}
Now, we have from \eqref{eq: integral form for P}
\begin{equation}
\begin{split}
    P_{x_1}(\xi,t) & - P_{x_2}(\xi,t)=  e^{-a_0 t}(P_{x_1} (\xi - x_1(t), 0) - P_{x_2} (\xi - x_2(t), 0))\\
    &+\displaystyle\frac{e^{-a_0 t}}{\ell}\displaystyle\int_0^t e^{a_0 s}(\omega_B(\xi + x_1(s)-x_1(t))-\omega_B(\xi + x_2(s)-x_2(t))\,ds.
\end{split}
\end{equation}
Thus, we have the estimate
\begin{equation*}
    \begin{split}
     \|P_{x_1}(\xi,t)& - P_{x_2}(\xi,t)\|_{L_{\xi}^{\infty}L_t^{\infty}}\\ 
    &\leq \|\partial_{\xi} P\|_{L_{\xi}^{\infty}} \|x_1(t)-x_2(t)\|_{L_t^{\infty}}\\
    &\quad\quad +\frac{1}{\ell}\left\| \displaystyle\int_0^t e^{a_0 (s-t)}\| \partial_{\xi} \omega_B\|_{L_{\xi}^{\infty}} | x_1(s)-x_1(t)- x_2(s)+x_2(t))|\,ds \right\|_{L_t^{\infty}}\\
    &\leq \left( \|\partial_{\xi} P\|_{L_{\xi}^{\infty}} + \frac{2}{a_0 \ell}\| \partial_{\xi} \omega_B\|_{L_{\xi}^{\infty}} \right) \|x_1(t)-x_2(t)\|_{L_t^{\infty}}.
 \end{split}
\end{equation*}
The same goes for the difference $Q_{x_1}(\xi,t) - Q_{x_2}(\xi,t)$, and we deduce
\begin{equation}
    \|P_{x_1}(\xi,t) - P_{x_2}(\xi,t)- Q_{x_1}(\xi,t) + Q_{x_2}(\xi,t)\|_{L_{\xi}^{\infty}L_t^{\infty}} \leq  C_1 \|x_1(t)-x_2(t)\|_{L_t^{\infty}}.
\end{equation}
with $C_1 = 4 \max\left\{\|\partial_{\xi}Q\|_{L_{\xi}^{\infty}},\|\partial_{\xi}P\|_{L_{\xi}^{\infty}},\frac{2}{a_0 \ell}\| \partial_{\xi}\omega_B\|_{L_{\xi}^{\infty}}\right\}$. 

We now deduce from \eqref{eq: F1-F2}
\begin{equation}
    \|F[x_1(\cdot)]-F[x_2(\cdot)]\|_{L_t^{\infty}} \leq \displaystyle \left( \frac{C_1\ell}{2\eta} \|\partial_{\xi} \Delta W\|_{L_{\xi}^{\infty}}+ \frac{k}{\eta}\right) \|x_1(t)-x_2(t)\|_{L_t^{\infty}},
\end{equation}
and obtain
\begin{equation}
    \|\psi[x_1(\cdot)]-\psi[x_2(\cdot)]\|_{L_t^{\infty}} \leq C_2 T \|x_1-x_2\|_{L_t^{\infty}},
\end{equation}
with $C_2 = \frac{C_1 \ell}{2\eta} \|\partial_{\xi}\Delta W\|_{L_{\xi}^{\infty}} +k/\eta$.

Taking $T = \frac{1}{ 2C_2}$, this proves that $\psi$ is a contraction, as claimed. Then, there exists a unique fixed point $\tilde x (\cdot) \in C^0([0,T])$ which satisfies \eqref{eq: 2 layers v}. \\

\noindent \textbf{Step 2: global solutions.}
The previous construction can be extended as long as $P_x$ and $Q_x$ remain bounded in $C^1([0,\ell])$ with respect to $\xi$ and in $L^{\infty}(\mathbb{R}_+)$ with respect to $t$, as shown by the formulas for $C_1$ and $C_2$. But from (\ref{eq: integral form for P}, \ref{eq: integral form for Q}) we have
\begin{equation*}
    \Vert P_x(\cdot, t) \Vert_{C_{\xi}^{1}} \leq e^{-a_0t} \Vert P \Vert_{C_{\xi}^{1}} + (1-e^{-a_0t})\frac{1}{a_0 \ell}\Vert \omega_B\Vert_{C_{\xi}^{1}},
\end{equation*}
which shows that 
\begin{equation*}
\Vert P_x \Vert_{C_{\xi}^{1}, L_{t}^{\infty}} \leq \max \left( \Vert P \Vert_{C_{\xi}^{1}},\frac{1}{a_0 \ell}\Vert \omega_B\Vert_{C_{\xi}^{1}}\right),
\end{equation*}
and the same goes for $Q_x$.
We thus obtain that there is a unique global solution $(x,\, P_x, \, Q_x)$ to (\ref{eq: 2 layers P Q}, \ref{eq: 2 layers v}) for all time $t\geq 0$.

In fact, $x\in C^1(\mathbb{R}_+)$ and $P_x, \, Q_x \in  C_{\#}^1([0,\ell] \times \mathbb{R}_+)$, with $\xi \mapsto P(\xi,\cdot)$ and $\xi \mapsto Q(\xi,\cdot)$ in $C_{\#}^1([0,\ell])$, as shown by the following bootstrap argument:
From equations (\ref{eq: integral form for P}, \ref{eq: integral form for Q}), we know that $P_x$ and $Q_x$ are continuous in time, which enables us to deduce that $F[x(\cdot)] \in C^0(\mathbb{R}_+)$. Therefore, $\psi[x(\cdot)] \in C^1(\mathbb{R}_+)$. But, since $x = \psi[x(\cdot)]$, we deduce that $x \in C^1(\mathbb{R}_+)$. Re-using equations (\ref{eq: integral form for P}, \ref{eq: integral form for Q}), we may then deduce that $P_x$ and $Q_x$ are in fact $C^1$ in time (and they actually have the minimal regularity of the initial conditions $P, \, Q$ and of $\omega_B$).

\subsection{Proof of Theorem \ref{th: Hopf bifurcation} (Hopf bifurcation)}

The functions $\xi \mapsto P(\xi,t)$ and $\xi \mapsto Q(\xi,t)$ are periodic with period $l$, and they can be then expanded in Fourier series
\begin{equation} \label{eq: fourier P and Q}
    \left\{\begin{array}{l}
P(\xi,t) = \displaystyle \frac{p_0(t)}{2} + \sum\limits_{n>0} \left( p^{c}_n(t) \cos{\frac{2n \pi \xi}{\ell}} + p^{s}_n(t)\sin{\frac{2n \pi \xi}{\ell}}\right), \\
Q(\xi,t) = \displaystyle \frac{q_0(t)}{2}+ \sum\limits_{n>0} \left( q^{c}_n (t)\cos{\frac{2n \pi \xi}{\ell}} + q^{s}_n(t)\sin{\frac{2n \pi \xi}{\ell}}\right).
\end{array}\right.
\end{equation}
We insert this expansion and the one for the transition rates \eqref{eq: om2, om1 general choices} into the system \eqref{eq: 2 layers full PDE}. By matching same order terms, we get an infinite number of ordinary differential equations for the coefficients of $P$ and $Q$. Namely, for $n = 0$ we get two decoupled equations for $p_0$ and $q_0$:
\begin{equation} \label{eq: p0,q0 th}
    \dot{p_0}(t) = - a_0(\Omega) p_0(t) + a_0(\Omega)/ 2 \ell,\quad \dot{q_0}(t) = - a_0(\Omega) q_0(t) + a_0(\Omega)/ 2\ell.
\end{equation}
For $n \neq 0$ we obtain:
\begin{equation} \label{eq: general pn qn th}
    \left\{\begin{array}{l}
\dot{p^{c}_n}(t) + \frac{2 \pi n }{\ell} v(t) p^{s}_n(t)= - a_0(\Omega) p^{c}_n(t) + a_n(\Omega)/ \ell\,,\\
\dot{p^{s}_n}(t) - \frac{2 \pi n }{\ell} v(t) p^{c}_n(t)= - a_0(\Omega) p^{s}_n(t) + b_n(\Omega)/ \ell\,,\\
\dot{q^{c}_n}(t) - \frac{2 \pi n }{\ell} v(t) q^{s}_n(t)= - a_0(\Omega) q^{c}_n(t) + a_n(\Omega)/ \ell\,,\\
\dot{q^{s}_n}(t) + \frac{2 \pi n }{\ell} v(t) q^{c}_n(t)= - a_0(\Omega) q^{s}_n(t) + b_n(\Omega)/ \ell\,,
\end{array}\right.
\end{equation}
together with the force balance equation
\begin{equation}\label{eq: force balance expansion}
      2 \eta \dot{x}(t) + 2 k x(t) + \pi U (p_1^{s}-q_1^{s}) = 0\,.
\end{equation}
Note that the coupling between the probabilities evolution and the force equilibrium equation takes place only for the first order coefficients. We are going to prove the existence of a Hopf bifurcation by treating order $n=0$ first, then $n=1$ and lastly $n>1$. Combining all three results will complete the proof. \\

\noindent {\textbf{Step 1.}} Zeroth order coefficients. It is clear that \eqref{eq: p0,q0 th} gives 
    \begin{equation*}
        p_0(t)=e^{-a_0t}\left(p_0(0)- \frac{1}{2l}\right) + \frac{1}{2\ell}\,,\quad q_0(t)=e^{-a_0t}\left(q_0(0)- \frac{1}{2l}\right) + \frac{1}{2\ell}\,,
    \end{equation*}
and both converge to $\frac{1}{2\ell}$ exponentially fast.\\

\noindent {\textbf{Step 2.}} First order coefficients.
We now prove the onset of oscillatory patterns for the first order coefficients $p_1^{c,s}(t)$, $q_1^{c,s}(t)$, and for the position $x(t)$. 

From \eqref{eq: general pn qn th} and \eqref{eq: force balance expansion} we obtain a first-order five-dimensional ODE:
\begin{equation} \label{eq: ODE p1 q1 th general}
    \left\{\begin{array}{l}
\displaystyle\dot{p^{c}_1} = - a_0(\Omega) p_1^c +   \big(\zeta x +   \displaystyle\frac{\lambda }{2} (p_1^{s}-q_1^{s})\big)p^{s}_1 + \frac{a_1(\Omega)}{\ell},\\[5pt]
\displaystyle\dot{p^{s}_1} = - a_0(\Omega) p_1^s - \big(\zeta  x +   \frac{\lambda }{2} (p_1^{s}-q_1^{s})\big)\displaystyle p^{c}_1  + \frac{b_1(\Omega)}{\ell},\\[5pt]
\dot{q^{c}_1}=  - a_0(\Omega) q_1^c - \big(\zeta x +   \displaystyle\frac{\lambda }{2} (p_1^{s}-q_1^{s})\big)\displaystyle q^{s}_1  + \frac{a_1(\Omega)}{\ell},\\[5pt]
\displaystyle\dot{q^{s}_1} = - a_0(\Omega) q_1^s + \big(\zeta  x +   \displaystyle\frac{\lambda }{2} (p_1^{s}-q_1^{s})\big) q^{c}_1  + \frac{b_1(\Omega)}{\ell},\\[5pt]
\dot{x} = -\displaystyle\frac{\ell}{2 \pi}\big(   \zeta  x +   \displaystyle\frac{\lambda }{2} (p_1^{s}-q_1^{s})\big)\,,
\end{array}\right.
\end{equation}
where $ \zeta = 2 \pi k / \eta \ell$ and $\lambda = 2 \pi^2 U / \eta \ell$.

We linearize the system around its equilibrium point $p_{eq}(\Omega) = (p_{1,\text{eq}}^{c}(\Omega),p_{1,\text{eq}}^{s}(\Omega),q_{1,\text{eq}}^{c}(\Omega),q_{1,\text{eq}}^{s}(\Omega), x_{\text{eq}}(\Omega))$, where
\begin{equation*}
 p_{1,eq}^c = q_{1,eq}^c = \frac{a_1(\Omega) }{ a_0(\Omega) \ell}\,,\,\,\, p_{1,eq}^s = q_{1,eq}^s = \frac{b_1(\Omega) }{ a_0(\Omega) \ell}\,,\,\,\, x_{eq}=0\,.
\end{equation*}
We observe that the Jacobian matrix has five eigenvalues: three of them are equal to 
$-a_0(\Omega)<0$ while the other two are of the form $\tau(\Omega) \pm \sqrt{-\frac{\zeta  \ell a_0(\Omega )}{2 \pi}+ \tau^2(\Omega)},$ where $\tau(\Omega)$ is given by \eqref{eq: tau}.
Since, by hypothesis, there exists a real and positive $\Omega=\Omega_0$ such that $\tau(\Omega_0)=0$, then $-  \zeta  \ell a_0(\Omega_0) +  2 \pi \,\tau^2(\Omega_0)<0$, and we deduce that the pair of complex eigenvalues can be written as $\tau(\Omega) \pm i \omega(\Omega)$ in a neighborhood of $\Omega_0$, where  
\begin{equation}
    \omega(\Omega) := \sqrt{\frac{\zeta  \ell a_0(\Omega )}{2 \pi}-  \tau^2(\Omega)},
    \label{eq: imaginary_smallomega}
\end{equation}
and they cross the imaginary axis at $\Omega=\Omega_0$.

In the following proposition, we are going to study the non-linear behavior of the vector field solution to \eqref{eq: ODE p1 q1 th general} by exploiting the center manifold theorem; the orbit structure near the fixed point and $\Omega_0$ is determined by the restriction of the non linear system to the center manifold. In particular, system \eqref{eq: ODE p1 q1 th general} restricted to the center manifold will show a super-critical Hopf bifurcation near $p_{eq}(\Omega_0)$ and $\Omega_0$.

\begin{proposition}[First order coefficients] \label{prop: first order coeff}
    With the hypothesis of Theorem \ref{th: Hopf bifurcation}, the non-linear system \eqref{eq: ODE p1 q1 th general} has a supercritical Hopf bifurcation near $(p_{eq}(\Omega_0),\Omega_0)$.
\end{proposition}

\begin{proof}

\noindent {\textbf{Change of variables.}}\\
In the first part of the proof we restrict the dynamical system to the center manifold. To compute the latter, we bring the system to a more suitable formulation.

Let us first transform the fixed point of \eqref{eq: ODE p1 q1 th general} to the origin. We define the new variables as
\begin{equation*}
    \delta p_1^c = p_1^c - p_{1,eq}^c,\, \delta p_1^s = p_1^s - p_{1,eq}^s,\,\delta q_1^c = q_1^c - q_{1,eq}^c, \, \delta q_1^s = q_1^s - q_{1,eq}^s.
\end{equation*}
 We then use the following linear and invertible change of variables
\begin{equation*}\label{eq: change of variable th}
    \begin{array}{cccc}
         r = \frac{a_1}{b_1}\delta p_1^{c} + \delta p_1^{s}, & 
         s = \frac{a_1}{b_1}\delta q_1^{c} +\delta q_1^{s}, &
         z = \frac{1}{2}(\delta p_1^{s} + \delta q_1^{s}), & 
         y = \frac{1}{2}(\delta p_1^{s} - \delta q_1^{s}),
    \end{array}
\end{equation*} 
and shorten the notation by taking $X=(r,s,z)^T$, and $Y=(y,x)^T$ to transform the system \eqref{eq: ODE p1 q1 th general} into
\begin{equation} \label{eq: 5d compact th general}
    \frac{d}{dt}\left(\begin{array}{l}
X\\
Y
\end{array}\right) = \mathbb{M}(\Omega) \left(\begin{array}{l}
X\\
Y
\end{array}\right) + \left(\begin{array}{c}
    G(X,Y)\\
    F(X,Y)
    \end{array}\right),
\end{equation}
where
\begin{equation*}
    \mathbb{M}(\Omega)=\left(
\begin{array}{ccccc}
 -a_0\text{Id}_{3,3} & 0\\
 0 & \mathbb{A}(\Omega)
\end{array}
\right)
\end{equation*}
\corrThm{is a block diagonal matrix with
\begin{equation}
     \mathbb{A}(\Omega)= \left(
\begin{array}{cc}
  -a_0 - \frac{\lambda}{\ell}\frac{{a_1}}{a_0} & \hspace*{0.5cm}- \frac{\zeta}{\ell}\frac{{a_1}}{a_0} \\\\
 -\frac{\lambda  \ell}{2 \pi } & \hspace*{0.5cm}-\frac{\zeta  \ell}{2 \pi } \\
\end{array}
\right),
\end{equation}
and the non linear part is defined as
\begin{equation*}
    G(X,Y,\Omega)=(\zeta   x+\lambda  y)\left(\begin{array}{c}
     \displaystyle\frac{{a_1}}{{b_1}} (y+z)+ \displaystyle\frac{{b_1}}{{a_1}}(-r+y+z)  \\[10pt]
   \displaystyle\frac{a_1}{b_1} (y-z)+ \frac{b_1}{a_1} (s+y-z)\\[10pt]
    -\displaystyle \frac{{b_1}}{{a_1}} \frac{ (r-s-2 y)}{2}
    \end{array}\right)\,,
\end{equation*}
and
\begin{equation}
    F(X,Y,\Omega)=(\zeta  x+\lambda  y)\left(\begin{array}{c}
     -\displaystyle\frac{{b_1}}{{a_1}}\frac{(r+s-2 z)}{2}\\
    0
    \end{array}\right),
\end{equation}
where we remind the reader that $a_0,a_1$ and $b_1$ depend on $\Omega$.}

\noindent{\textbf{Computation of the center manifold.}}

The center manifold can then be computed by using standard techniques (see \cite{Wiggins}, Chapter 20, Section 2). We start by rewriting the system in such a way that $\Omega_0$ is moved to the origin through the change of variable $\delta\Omega\,= \Omega - \Omega_0$. As it is classical, we treat $\delta\Omega\,$ as a variable of the system. This means that we add the equation $\dot {\delta\Omega\,}=0$ to the dynamical system and that the non linear part of the system includes all the products $\delta\Omega\, r$, $\delta\Omega\, s$, $\delta\Omega\, z$ etc. 

\corrThm{Since the terms in the matrix $\mathbb{M}$ are nonlinear in $\Omega$, we expand them near $\Omega_0$ as for instance
\begin{equation}\label{eq: Omega expansion}
a_0(\Omega_0+\delta \Omega) = a_0(\Omega_0)+a'_0(\Omega_0)\delta \Omega + O(\delta\Omega^2).
\end{equation}

Using such expansions in the system \eqref{eq: 5d compact th general}, we obtain
\begin{equation} \label{eq: 5d compact th general centered in zero}
    \frac{d}{dt}\left(\begin{array}{c}
X\\
Y\\
\delta\Omega
\end{array}\right) = \mathbb{M}(\Omega_0)\left(\begin{array}{c}
X\\
Y\\
0
\end{array}\right) + \left(\begin{array}{c}
g(X,Y,\delta\Omega)\\
f(X,Y,\delta\Omega)\\
0
\end{array}\right),
\end{equation}
with \[g(X,Y,\delta\Omega)= -a'_0(\Omega_0)\delta \Omega X +G(X,Y,\Omega_0) + O(|\delta\Omega|^3+ |X|^3)\]
and
\begin{equation*}
\begin{split}
    f(X,Y,\delta\Omega) &= -\left(\begin{array}{c}
\left(a'_0(\Omega_0)  +\displaystyle\frac{\lambda}{\ell}\left(\frac{{a_1}}{a_0}\right)'(\Omega_0)\right)\delta\Omega\, y +\displaystyle \frac{\zeta}{\ell} \left(\frac{{a_1}}{a_0}\right)'(\Omega_0)\delta\Omega\, x\\\\
    0
    \end{array}\right)\\
    &\hspace*{5cm}+ F(X,Y,\Omega_0)+ O(|\delta\Omega|^3+ |X|^3).
\end{split}
\end{equation*} }

We can define a center manifold as 
\begin{equation*}
    W^c(0) = \{(X,Y,\delta\Omega): X= \mathbf{h}(Y,\delta\Omega),
    |Y|<\eps, |\delta\Omega\,|<\bar{\eps}, \mathbf{h}(0,0)=0,\,\nabla \mathbf{h}(0,0)=0\},
\end{equation*}
for $\eps$ and $\bar{\eps}$ sufficiently small and $\mathbf{h}=(h_1,h_2,h_3)$ smooth enough.

In order for $\mathbf{h}$ to be the center manifold for the system \eqref{eq: 5d compact th general centered in zero}, the graph of $\mathbf{h}(Y,\delta\Omega)$ has to be invariant under the dynamics generated by \eqref{eq: 5d compact th general centered in zero}. Hence, by plugging $\mathbf{h}$ into the system and compute its derivative, we get the following quasilinear differential equation 
\begin{equation}\label{eq: N(h)=0 general}
    \nabla_{Y} \mathbf{h} \cdot \left(\mathbb{A}(\Omega_0)Y+ f(\mathbf{h},Y,\delta\Omega)\right)=-a_0(\Omega_0)\mathbf{h} + g(\mathbf{h},Y,\delta\Omega).
\end{equation}

The solution $\mathbf{h}(Y,\delta\Omega)$ of \eqref{eq: N(h)=0 general} can be approximated with a power series expansion up to any desired degree of accuracy. In our case we expand them up to order two defining, for $i=1,2,3$ and $Y=(y,x)^T$
\begin{equation}\label{h_i}
    h_i(Y,\delta\Omega):= a_{i1}y^2+a_{i2} y x+a_{i3} y \delta\Omega\, +a_{i4} x^2+a_{i5} x \delta\Omega\, +a_{i6} \delta\Omega\, ^2 + O(|\delta\Omega|^3+ |X|^3).
\end{equation}
To completely determine (locally) the center manifold, we have to compute the coefficients $a_{ij}$ knowing that the functions defined in \eqref{h_i} must solve \eqref{eq: N(h)=0 general}. For the detailed computations, we refer the reader to Appendix \ref{appendix: center manifold computations}.

Restricted to the center manifold, the original system \eqref{eq: 5d compact th general} has the following form:
\begin{equation} \label{eq: Hopf restricted to CM}
    \frac{d}{dt}Y = \mathbb{A}(\Omega) Y + f(\mathbf{h}(Y),Y)+O(|\delta\Omega|^3+ |X|^3)
\end{equation}
where
\corrThm{\begin{equation*}
 f(\mathbf{h}(Y),Y)=
    \left(\begin{array}{c}
\displaystyle-\frac{2 \pi  a_0(\Omega_0) \left(\zeta ^3 x^3+\lambda ^3 y^3\right)+\zeta  \ell (\zeta  x+\lambda  y)^3}{2 a_0(\Omega_0) \lambda  (\pi  a_0(\Omega_0)+2 \zeta  \ell)}\\
    0
    \end{array}\right)+O(|\delta\Omega|^3+ |X|^3).
\end{equation*}}

Observe that $\tau(\Omega)  \pm i \omega(\Omega) $ is the pair of conjugated eigenvalues of $\mathbb{A}(\Omega)$ and therefore we are in the hypothesis of the existence of a Hopf bifurcation for a two-dimensional system.\\

\noindent {\textbf{Normal form.}}

The next step is to bring system \eqref{eq: Hopf restricted to CM} into its normal form, from which we deduce the type of Hopf bifurcation that the system is attaining.

In order to proceed, we apply a further change of coordinates, such that $\mathbb{A}(\Omega)$ is transformed to its real Jordan form. Namely, we define the transformation matrix
\begin{equation*}
    \mathbb{P}=\begin{pmatrix}
        P_{11}&\quad& P_{12}\\
        1 &\quad&0
    \end{pmatrix},
\end{equation*}
where 
\corrThm{\[P_{11}=\displaystyle\frac{ \pi }{\lambda  \ell} \left(\frac{\lambda}{\ell}\frac{{a_1}}{a_0}+a_0\right)-\frac{\zeta}{2 \lambda  }\] and \[P_{12} = \displaystyle\frac{\sqrt{-4 \pi  \lambda  \ell a_0 a_1 (2 \pi  a_0+\zeta  \ell)-\ell^2 a_0^2 (\zeta  \ell-2 \pi  a_0)^2-4 \pi ^2 \lambda ^2 a_1^2}}{2 \lambda  \ell^2 a_0}. \]}
The matrix $\mathbb{P}$ defines the new coordinates $Y= \mathbb{P}\tilde{Y}$ thanks to which equation \eqref{eq: Hopf restricted to CM} can be expressed in its normal form
\begin{equation} \label{eq: Hopf ODE final}
     \frac{d}{dt}\tilde{Y} = \begin{pmatrix}
    \tau(\Omega) &-\omega(\Omega)\\
        \omega(\Omega) &\tau(\Omega)
\end{pmatrix} \tilde{Y} + \mathbb{P}^{-1}f(\mathbf{h}(\mathbb{P}\tilde Y),\mathbb{P}\tilde Y)+O(\|\delta\Omega\|^3+ \|Y\|^3).
\end{equation}

To compute key properties of the system, such as the amplitude of the limit cycles, we change coordinates to the polar ones ${\tilde Y}^T= \rho (\sin(\theta),\cos(\theta))$, and get
\begin{equation}\label{eq: 5d normal form polar coordinates}
    \left\{\begin{array}{l}
         \dot{\rho}(t) = \tau'(\Omega_0) \delta \Omega \rho(t)+ \tilde{\tau} \rho^3(t) + O(\delta \Omega^2\rho,\delta \Omega\rho^3),  \\
         \dot{\theta}(t) =\omega(\Omega_0) + \omega'(\Omega_0) \delta \Omega +  \tilde{\omega}\rho^2(t) +O(\delta \Omega^2,\delta \Omega\rho^2).
    \end{array}\right.,
\end{equation}
that is the normal form of \eqref{eq: Hopf restricted to CM} in polar coordinates around $\Omega_0$.

The non linear part $\mathbb{P}^{-1}f(\mathbf{h}(\mathbb{P}\tilde Y),\mathbb{P}\tilde Y)$ determines the constants $\tilde{\tau}$ and $\tilde{\omega}$. The first one is involved in the expression for the limit cycle amplitude, and we compute it using a well know formula (see Appendix \ref{app: tilde tau}). \corrThm{We obtain
\begin{equation*}
\tilde{\tau} = - \frac{3 \pi \zeta}{4 \ell} \left(\frac{\pi a_0(\Omega_0) + \ell \zeta}{\pi a_0(\Omega_0) + 2 \ell \zeta}\right)\,.
\end{equation*}

Since $\tilde{\tau}$ is negative and
\begin{equation*}
    \tau'(\Omega_0) = -\frac{a_0'(\Omega_0) \left(\ell a_0(\Omega_0)^2-\lambda  a_1(\Omega_0)\right)+\lambda  a_0(\Omega_0) a_1'(\Omega_0 )}{2 \ell a_0(\Omega_0)^2},
\end{equation*}}\,\\
is positive by hypothesis, the reduced system \eqref{eq: Hopf restricted to CM}, and hence, the whole system \eqref{eq: 5d compact th general}, shows a supercritical Hopf bifurcation near the bifurcation parameter $\Omega_0$.

In particular, for $\Omega$ sufficiently near and greater then $\Omega_0$ there exists an asymptotically stable periodic orbit with radius and \corrThm{frequency as in \eqref{eq: th amplitude Hopf} and \eqref{th:freq}, respectively}.

This finishes the proof of Proposition \ref{prop: first order coeff}
\end{proof}

\noindent {\textbf{Step 3.}} Higher order terms. Lastly, we are going to prove that, after large times, the solutions of \eqref{eq: general pn qn th} $p_n^{c,s}(t)$ and $q_n^{c,s}(t)$, with $n>1$ are periodic with the same period as $x(t)$. We remark again that, since the dynamics of $x$ \eqref{eq: force balance expansion} is independent of $p_n^c, \, p_n^s, \,  q_n^c,\,  q_n^s$ for $n > 1$, we may assume that $x(t)$ is given and periodic, and split the system into pairs of coupled equations as stated in the following Proposition.
\begin{proposition}[Higher order coefficients]\label{th: n order coeff}
   For $n > 1$ consider the infinite system of four equations:
\begin{equation}
    \left\{\begin{array}{l}
\dot{p^{c}_n}(t) + \displaystyle \frac{2 \pi n}{\ell} \dot{x} p^{s}_n(t)= - a_0(\Omega) p^{c}_n(t) + a_n(\Omega)/ \ell,\\
\dot{p^{s}_n}(t) - \displaystyle \frac{2 \pi n}{\ell} \dot{x} p^{c}_n(t)= - a_0(\Omega) p^{s}_n(t) + b_n(\Omega)/ \ell.
\end{array}\right.
\label{eq:pn}
\end{equation}
and
\begin{equation}
    \left\{\begin{array}{l}
\dot{q}^c_n(t) - \displaystyle\frac{2 \pi n}{\ell} \dot{x} q^{s}_n(t)= - a_0(\Omega) q^{c}_n(t) + a_n(\Omega)/ \ell,\\
\dot{q}^{s}_n(t) + \displaystyle\frac{2 \pi n}{\ell} \dot{x} q^{c}_n(t)= - a_0(\Omega) q^{s}_n(t) + b_n(\Omega)/ \ell.
\end{array}\right.
\label{eq:qn}
\end{equation}
Suppose that, for $t \geq 0$, the function $t \mapsto x(t)$ is periodic of period $T$. Then, the solutions to the systems \eqref{eq:pn} and \eqref{eq:qn} may have two behaviors: they either converge in time to periodic solutions of the same period $T$ if $n$ is odd, or they go to zero for large times when $n$ is even. 
\end{proposition}
\begin{proof}
We introduce $z_n = p_n^c + ip_n^s$.
The first two equations of \eqref{eq:pn} then become: 
\begin{equation}
    \dot{z}_n + \left(a_0(\Omega) -\frac{2in\pi}{\ell}\dot x\right)z_n = c_n
\end{equation}
where $\displaystyle c_n = \frac{a_n(\Omega) + ib_n(\Omega)}{\ell}$ is constant. We thus deduce the following expression for $z_n(t)$
\begin{equation} \label{eq: zn(t)}
    z_n(t)  = \displaystyle c_n \int\limits_0^te^{-\int\limits_u^t(a_0(\Omega) - i\bar v(s))ds}du + z_n(0)e^{-\int\limits_0^t (a_0(\Omega)-i\bar v(s))ds}\,,
\end{equation}
where $\bar{v} = \frac{2n\pi}{\ell}\dot{x}$.

Let us first notice that the second term, $ \displaystyle z_n(0)e^{-\int\limits_0^t (a_0(\Omega)-i\bar v(s))ds}$ goes to zero when $t$ goes to infinity. This, together with the fact that $a_n=b_n=0$ if $n$ is even (see \eqref{eq: om2, om1 general choices}), gives $c_n=0$ and the solutions $p_n^{c,s}$ go to zero for large times when $n$ is even. 

Let us now focus on the case where $n$ is odd. We want to prove that, when $\bar v$ is periodic of period $T$, $z_n$ converges towards a $T$-periodic solution after a transitory regime. We again notice that the second term of 
\eqref{eq: zn(t)} converges to 0, and the result will follow by studying only the first term.

We thus denote by $\tilde{z}_n(t)= \int\limits_0^te^{-\int\limits_u^t(a_0 - i\bar v(s))ds}du$, and compute
\begin{equation}
    \begin{array}{lll}
         \displaystyle \tilde{z}_n(T+t)- \tilde{z}_n(t) & =  \int\limits_0^{t+T}e^{-\int\limits_u^{t+T}(a_0 - i\bar v(s))ds}du - \int\limits_0^te^{-\int\limits_u^t(a_0 - i\bar v(s))ds}du \\
         & =  \displaystyle \int\limits_0^Te^{-\int\limits_u^{t+T}(a_0 - i\bar v(s))ds}du + \displaystyle \int\limits_T^{t+T}e^{-\int\limits_u^{t+T}(a_0 - i\bar v(s))ds}du \\&\quad-\int\limits_0^te^{-\int\limits_u^t(a_0 - i\bar v(s))ds}du\,.
    \end{array}
\end{equation}
Using the fact that $\bar{v}$ is $T$-periodic, the last two terms cancel and we obtain
\begin{eqnarray*}
 \tilde{z}_n(T+t)- \tilde{z}_n(t)&=& \displaystyle \int\limits_0^Te^{-\int\limits_u^{t+T}(a_0 - i\bar v(s))ds}du\\
 &=& e^{-(t+T)a_0}\displaystyle \int\limits_0^Te^{ua_0 + \int\limits_u^{t+T}i\bar v(s)ds}du
\end{eqnarray*}
which goes to $0$ when $t$ goes to infinity. 

We therefore proved that $z_n$ converges to a $T$-periodic function.
The same result holds for $(q^c_n,q^s_n)$ solution to \eqref{eq:qn}.
\end{proof}

Finally, if we have initial conditions that are close to the equilibrium point of system \eqref{eq: 2 layers full PDE}, and an amount of ATP $\Omega$, which is close to $\Omega_0$, we know exactly how the solution of the system evolves in time. The first order coefficients of \eqref{eq: fourier P and Q} go exponentially fast to the constant term $1/\ell$, the first order coefficients starts to oscillate in a limit cycle, and the higher order terms either go to zero or have the same patterns as $p_1^{c,s}$ and $q_1^{c,s}$ with same period of oscillation. Globally, the solution of \eqref{eq: 2 layers full PDE}, shows a supercritical Hopf-bifurcation in time, close enough to the parameter $\Omega_0$.\\

\corrThm{\textbf{Comment on hypotheses in Theorem \ref{th: Hopf bifurcation}.}
We have shown that if there exists a value $\Omega_0 > 0$ such that $\tau(\Omega_0) = 0$ and $\tau'(\Omega_0) > 0$, then the system undergoes a supercritical Hopf bifurcation. The coefficients $a_0$ and $a_1$ have to be suitably chosen to satisfy these hypotheses. 

In particular, to satisfy the first condition, the coefficient $a_1(\Omega_0)$ must be negative since it is given by:
\begin{equation} \label{a1}
    a_1(\Omega_0) = -\frac{\ell}{2 \lambda} \left[ 2 a_0^2(\Omega_0) + \frac{\zeta \ell}{\pi} a_0(\Omega_0) \right] < 0,
\end{equation}
with $a_0(\Omega_0)$ positive.

Concerning the second condition, note that it has a physical meaning since increasing ATP concentration $\Omega$ drives the filament configuration from a stable to an oscillatory regime. Notice that, if $\tau'(\Omega_0) < 0$, the system still undergoes a supercritical Hopf bifurcation, but in this case, the stable periodic orbit emerges as $\Omega$ decreases, which contradicts the hypothesis that $\Omega$ is the ATP concentration. This second condition is equivalent to
\begin{equation}
    a_1'(\Omega_0) > -\frac{\zeta \ell^2}{2 \lambda \pi} a_0'(\Omega_0),
\end{equation}
where equation \eqref{a1} has been used to substitute $a_1(\Omega_0)$ in the derivative $\tau'(\Omega_0)$.

For example, if we take $a_0(\Omega)=c>0$, then $a_1(\Omega_0) = -\frac{\ell}{2 \lambda} \left[ 2 c^2 + \frac{\zeta \ell}{\pi} c \right]=C$. We can then make a simple choice for $a_1(\Omega)$, taking an affine function $a_1(\Omega)= B (\Omega-\Omega_0) + C$, where $B$ is a constant. Depending on the sign of $B$ we are in the hypothesis of Theorem \ref{th: Hopf bifurcation} (if $B< 0$) or not (if $B>0$).
}

\corr{\subsection{$N$-row model: theoretical considerations}} \label{sec: Nbifurcation}
In order to have a general idea of the system \eqref{eq: Nlayers} dynamics, as in \eqref{eq: fourier P and Q}, we expand in Fourier series $Q_1, \dots, Q_N$ and treat order by order the Fourier coefficients of the probabilities. Then, to investigate the existence of a Hopf bifurcation we consider its linearization, as it was done with the two-row model.
For the first order coefficients, we obtain a $3N-1$ dimensional ODE system with unknowns: \[p_j^{c,s},\quad j=1,\dots,N,\quad\quad x_k,\, k=1,\dots,N-1.  \]
We then perform some linear change of variables, similarly to \eqref{eq: change of variable th}, in order to write the Jacobian of the system as a block matrix. The first block is a $N+1$ diagonal matrix with real and negative entries $-a_0(\Omega)$. The rest of the linearized system is defined by the following equations 
\begin{equation} \label{eq: N-1 copies}
\left(\begin{array}{c}
     \dot q_k   \\
    \dot w_k \\
\end{array}\right)=
    \begin{pmatrix}
         - \frac{2 \pi^2}{\ell \eta} U p_e - a_0 \quad& - \frac{2 \pi}{\ell \eta} k p_e \\
        - \pi U /\eta & - k / \eta \\
    \end{pmatrix} \left(\begin{array}{c}
         q_k   \\
     w_k\\
\end{array}\right)
\end{equation}
for $k=1,\dots,N-1$, where  $p_e = a_1/(a_0 \ell)$, $q_k = p_k^s - p_{k+1}^s$ for $k=1,\dots,N-1$, and 
\begin{equation}
    \left(\begin{array}{c}
           w_1\\
            \vdots\\
            w_{N-1}
    \end{array}\right) =
     \begin{pmatrix}
    2 & -1 & & & \\
    -1 & 2 & -1 & & \\
    & & \ddots & & \\
    & & -1 & 2 & -1 \\
    & & & -1 & 2
  \end{pmatrix} 
    \left(\begin{array}{c}
           x_1\\
            \vdots\\
            x_{N-1}
    \end{array}\right).
\end{equation}

It follows that the Jacobian has $N+1$ real and negative eigenvalues $-a_0(\Omega)$, and $N-1$ identical pairs of complex and conjugated ones $\mu_k(\Omega)$ for $k=1,\dots, N-1$ which come from \eqref{eq: N-1 copies}. We observe that $\mu_k(\Omega)=\tau(\Omega) + i \omega(\Omega)$, with $\tau$ and $\omega$ defined as in Theorem \ref{th: Hopf bifurcation} by equations \eqref{eq: tau} and \eqref{eq: imaginary_smallomega}. 
Then, if there exists $\Omega=\Omega_0$ such that $\tau(\Omega_0)=0$ and $\tau'(\Omega_0)>0$, the eigenvalues cross the imaginary axis all at the same bifurcation parameter. Thus, there is a suggestion of a bifurcation in the dynamics at $\Omega_0$, indicated by the linear part of the complete dynamical system. We will not investigate the theoretical aspects of this model further in this paper, as the Central Manifold Theorem used in previous section, does not apply to such systems. However, as shown in the next section, numerical simulations still suggest a potential pattern in the oscillations and hint that there is indeed a bifurcation.\\

\section{Numerics}\label{sec : numerical_scheme}
We now describe the numerical scheme used throughout the remaining of the paper. We only write it for $N=2$, but it can easily be extended to $N$-row models, for $N>2$ or restricted to the $1$-row model.

We write a first-order upwind scheme for the densities $P$ and $Q$ where, after each step, we update the velocity $v$. Let us take $\Delta x$ such that $\ell = J \Delta x$ for $J \in \mathbb{N}^*$, and $\Delta t$ the time step. We define $P_j^n := P(j \Delta x, n \Delta t)$ and $Q_j^n := Q(j \Delta x, n \Delta t)$ for $n\geq 1$ and $1 \leq j \leq J$. Being $\ell$-periodic, we may extend $P_j^n$ for all $j \in \mathbb{Z}$ by setting $P_{j+J}^n=P_j^n$ for all $j \in \mathbb{Z}$ and all $n \in \mathbb{N}$; the same can be done for $Q_j^n$. At a fixed time step $t= n \Delta t$, $n\geq 1$, the velocity $v^n$ of the central tubule pair being known,  we compute, for $j \in \mathbb{Z}$
\begin{align}
\left\{\begin{array}{lllll}
    \displaystyle S^n_j = -a_0 P^n_j + \omega_B(j\Delta x)/\ell, & \\\\
    \displaystyle T^n_j = -a_0 Q^n_j + \omega_B(j\Delta x)/\ell, & \\\\
    P^{n+1}_j = \left\{\begin{array}{lll} \displaystyle \left( 1 - v^{n} \frac{\Delta t}{\Delta x}\right) + v^{n} \frac{\Delta t}{\Delta x}P^n_{j-1} + \Delta t S^n_j, & \text{if } v^{n} >0,\\
     \displaystyle \left( 1 + v^{n} \frac{\Delta t}{\Delta x}\right) - v^{n} \frac{\Delta t}{\Delta x}P^n_{j+1} + \Delta t S^n_j, & \text{if } v^{n} <0,\end{array}\right.\\\\
    Q^{n+1}_j = \left\{\begin{array}{ll}\displaystyle \left( 1 + v^{n} \frac{\Delta t}{\Delta x}\right) - v^{n} \frac{\Delta t}{\Delta x}Q^n_{j-1} + \Delta t T^n_j, & \text{if } v^{n} >0,\\
     \displaystyle \left( 1 - v^{n} \frac{\Delta t}{\Delta x}\right) + v^{n} \frac{\Delta t}{\Delta x}Q^n_{j+1} + \Delta t T^n_j, & \text{if } v^{n} <0,\end{array}\right.
\end{array}\right.
\label{eq: upwind}
\end{align}
and $(x^{n+1}, \,v^{n+1})$ are
\begin{align}
\left\{\begin{array}{ll}
    \eta v^{n+1} = \Delta x \sum\limits_{j=1}^J \left(\left(P^{n+1}_j - Q^{n+1}_j\right)\partial_{\xi} \Delta W(j\Delta x)\right) - kx^{n},\\
    x^{n+1} = x^{n} + v^{n+1} \Delta t.
\end{array}\right.
\label{eq: speedupwind}
\end{align}
Notice that, in \eqref{eq: upwind}, only one of the two updates is used for all $j$. Moreover, due to the stability of the upwind scheme, we need to check at each iteration that $\displaystyle \left| v^n \frac{\Delta t}{\Delta x}\right| <1$.

\corr{\subsection{Choice for parameters}\label{sec: alpha}

In order to implement the numerical simulations, we have to characterize $\omega_A$ and $\omega_B$ by fixing their coefficients. For simplicity, we impose that $a_n=b_n=0$ for $n>1$ since, as observed in Theorem \ref{th: Hopf bifurcation}, they have no influence over the displacement's oscillations. We choose $b_1 \neq 0$ to make the transition rates not-symmetric, as in \cite{guerin2011dynamical}.
\corrThm{We opted for particularly simple polynomial functions, defining \[a_0(\Omega) = c\Omega^{1/2},\quad\quad a_1(\Omega)=b_1(\Omega)= d\Omega,\] assuming $c$ and $\Omega_0$ are chosen. We analytically recover $d$ by imposing $\tau(\Omega_0)= 0$ 
\begin{equation}
d = d(\Omega_0) = -\frac{c \ell}{2 \pi \lambda} \left(2 \pi  c +\frac{\zeta  \ell}{\sqrt{\Omega_0}}\right)
\end{equation}
in such a way that the condition $\tau(\Omega_0)=0$ is satisfied.
By substituting $d$ into $\tau'(\Omega)$ we get the other condition $\tau'(\Omega_0) = \frac{\zeta  \ell}{8 \pi  \Omega_0}>0$.} 
 
Using the parameter values presented in Table 1, we get $d=d(\Omega_0) = -56.4588 \, s^{-1}$. Note that the parameters shown in Table 1, match F. Jülicher's \cite{camalet2000generic}. We have taken for $k_BT$ the value at $36^{\circ}C$ (human body temperature).

\begin{table}[ht!]
\centering
\caption{Values of the parameters used in numerical simulations.}
{\begin{tabular}{@{}cc@{}} 
\toprule
Parameter & Value \\
 \hline
        $\ell$ &  $10 \, nm$\\
        $k_BT$ &  $4.2668\cdot 10^{-3} \, nN \cdot nm$\\
        $\eta$ & $1.0 \cdot 10^{-7} \, kg/s$\\
        $k$ & $9.5 \cdot 10^{-5} \, kg/s^2$\\
        $U$ & $10 k_BT$\\
        $a_0^0 $ & $1.0 \cdot 10^{3} \, s^{-1}$\\
         $\Omega_0$ & $15 k_BT$\\ 
         \hline
\end{tabular}}
\end{table}
 
In the following, all simulations are carried out defining $\Delta W(\xi) = U \cos (2 \pi \xi / \ell)$.}

\subsection{$2$-row: numerical simulations}

\corr{Using the upwind scheme (\ref{eq: upwind}-\ref{eq: speedupwind}) and the parameters chosen above, we run simulations for the two-row model. We compare its behavior to the one of the one-row model~\eqref{eq: full pde 1row} for which the simulations are carried out using an obvious modification of the scheme.}

When the ATP concentration $\Omega$ is lower than $\Omega_0$, the central pair does not move from its zero equilibrium position. On the contrary, when the ATP concentration is high enough, we observe an oscillatory displacement between microtubules and in the probabilities. As shown in Figure \ref{fig: two-row simulations P1,P2, DeltaP, x}, the behavior of the two-row model with respect to the ATP concentration is comparable with the one for the one-row model.
 The main difference between the one-row model and the two-row one concerns the equilibrium position around which the solution $x(t)$ oscillates, as expected. This is clearly illustrated in Figure \ref{fig: two-row simulations P1,P2, DeltaP, x}(b), where the displacement takes place around zero, which is coherent with the symmetry of the physical problem. 
 


  \begin{figure}[ht!]
\centering
  \subfloat[]{\includegraphics[width=0.5\linewidth]{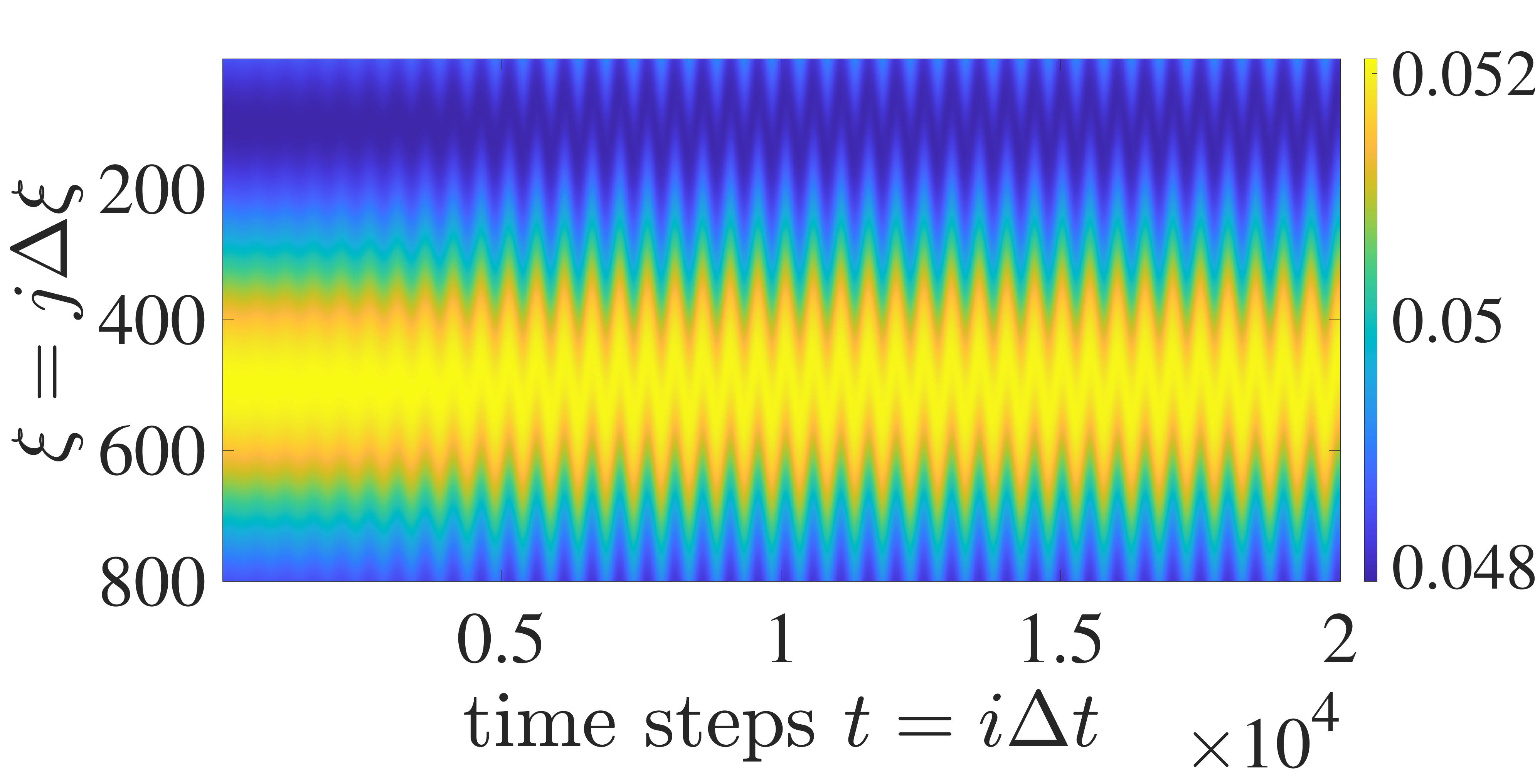}}
\hfil
    \subfloat[]{\includegraphics[width=0.5\linewidth]{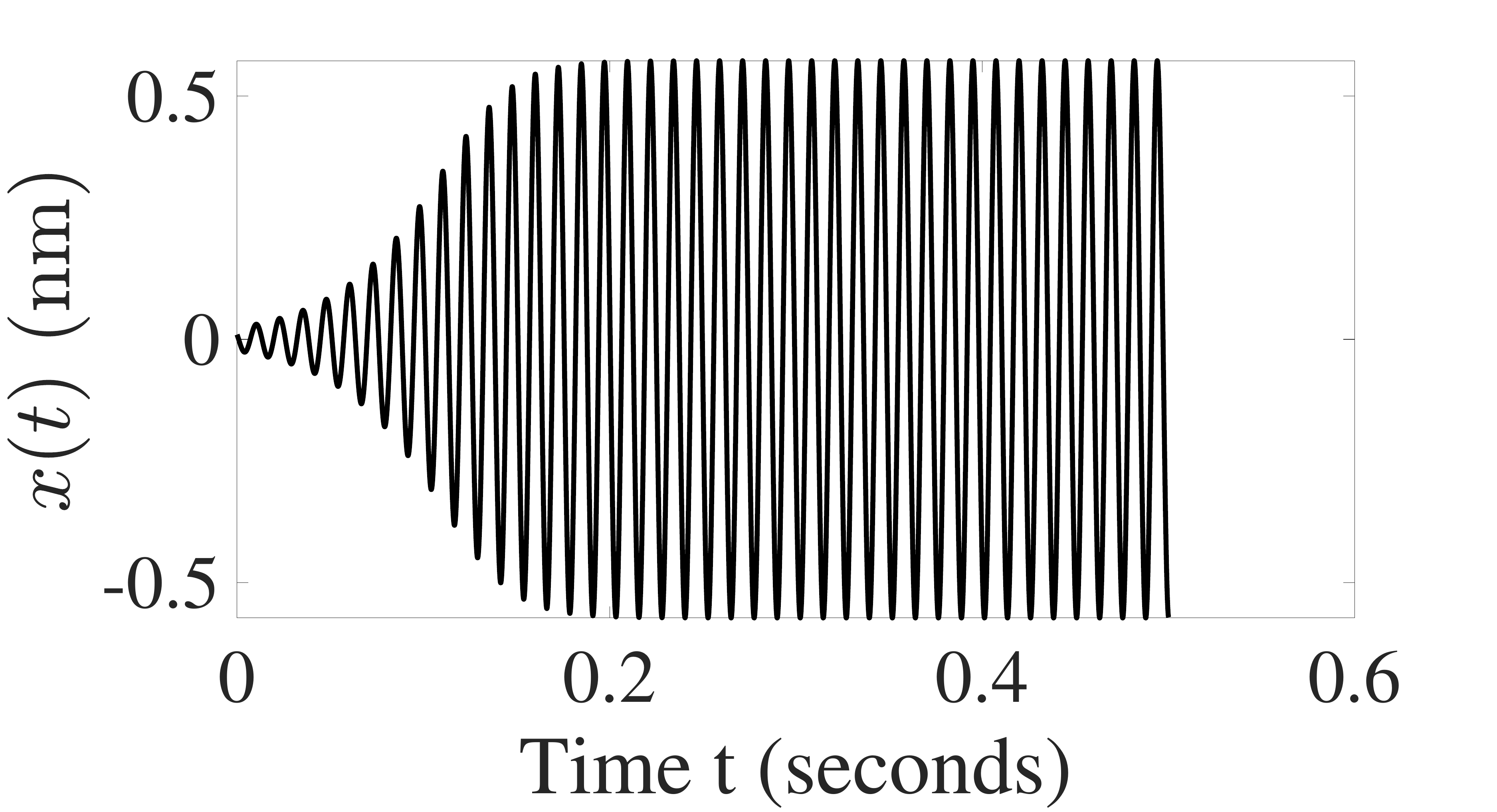}}

\centering
  \subfloat[]{\includegraphics[width=0.5\linewidth]{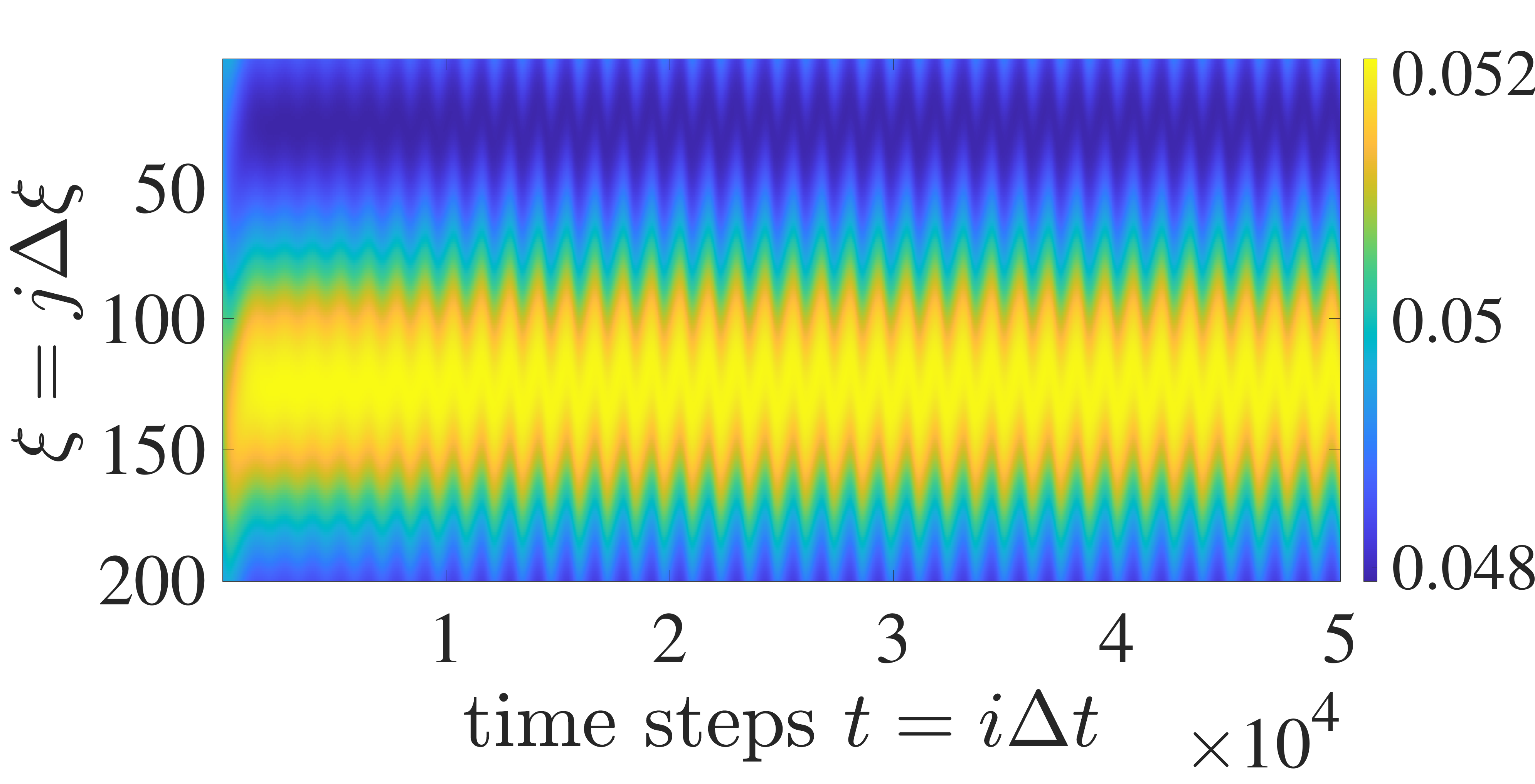}}
\hfil
    \subfloat[]{\includegraphics[width=0.5\linewidth]{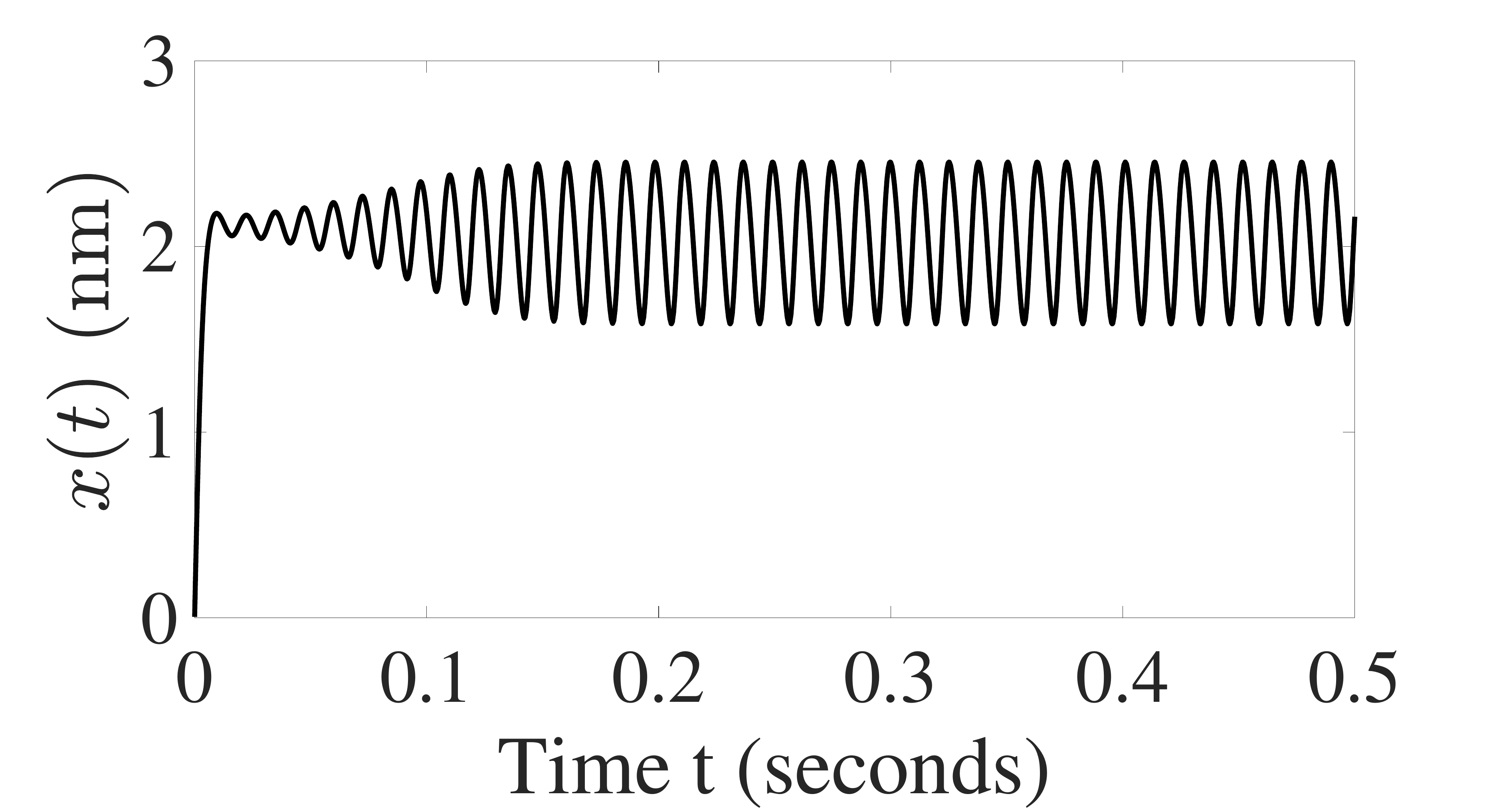}}
  \caption{Two-row model ((a),(b)) and one-row model ((c),(d)): (a) Probability density $P_1(\xi,t)$ of the top layer of being in state $1$ over time in a periodicity cell, in steady-state regime. The probability density $P_2(\xi,t)$ of the bottom layer of being in state $1$ is similar to $P_1(\xi, t)$. (b) Relative tubule shift $x$ over time. (c) Probability density $P(\xi,t)$ of the moving layer in the one-row model. (d) Relative tubule shift $x$ over time in the one-row model.}

    \label{fig: two-row simulations P1,P2, DeltaP, x}
    \end{figure}

\subsubsection{\corr{Amplitude of the oscillations}}
\corr{In this section, we study the evolution of the amplitude when the ATP concentration varies. In particular, we compare the \textit{theoretical amplitude} given by the asymptotic formula \eqref{eq: th amplitude Hopf} with the ones observed for the numerical solutions of the ODE system~\eqref{eq: ODE p1 q1 th general}, solved using MATLAB’s \texttt{ode45}, and of the the full PDE model \eqref{eq: 2 layers full PDE}, computed with the upwind scheme. In view of the shape of the solution, to determine the \textit{numerical amplitude} for both the ODE and the PDE, we calculate half of the difference between the maximum and minimum values of the displacement $x(t)$ along time, getting rid of the first 1000 time steps.} 

The parameter $\delta$ indicates the relative distance from the instability $ \Omega_0$; therefore, the point where the simulations is performed is $\Omega(\delta)=\Omega_0(1+\delta)$.
For this choice of parameters and neglecting higher-order terms, the theoretical amplitude \eqref{eq: th amplitude Hopf} simplifies to:
\begin{equation}\label{eq: th amplitude Hopf parameters choice}
    \rho(\delta)=\rho(\Omega(\delta))= \sqrt{\delta \frac{ \ell^2}{6 \pi^2} \left(\frac{\pi a_0^0\Omega_0^{1/2} + 2 \ell \zeta}{\pi a_0^0\Omega_0^{1/2} + \ell \zeta}\right)}.
\end{equation}
\\\\

In Figure \ref{fig:ampl_ODE_2rows_err_2rows }, simulations only start after $\delta=0.05$. This is related to the fact that the closer we are to $\Omega_0$, the slower the solution enters into its limit cycle. 
\begin{figure}[hbtp!]
\centering
    \subfloat[]{\includegraphics[width=0.5\linewidth]{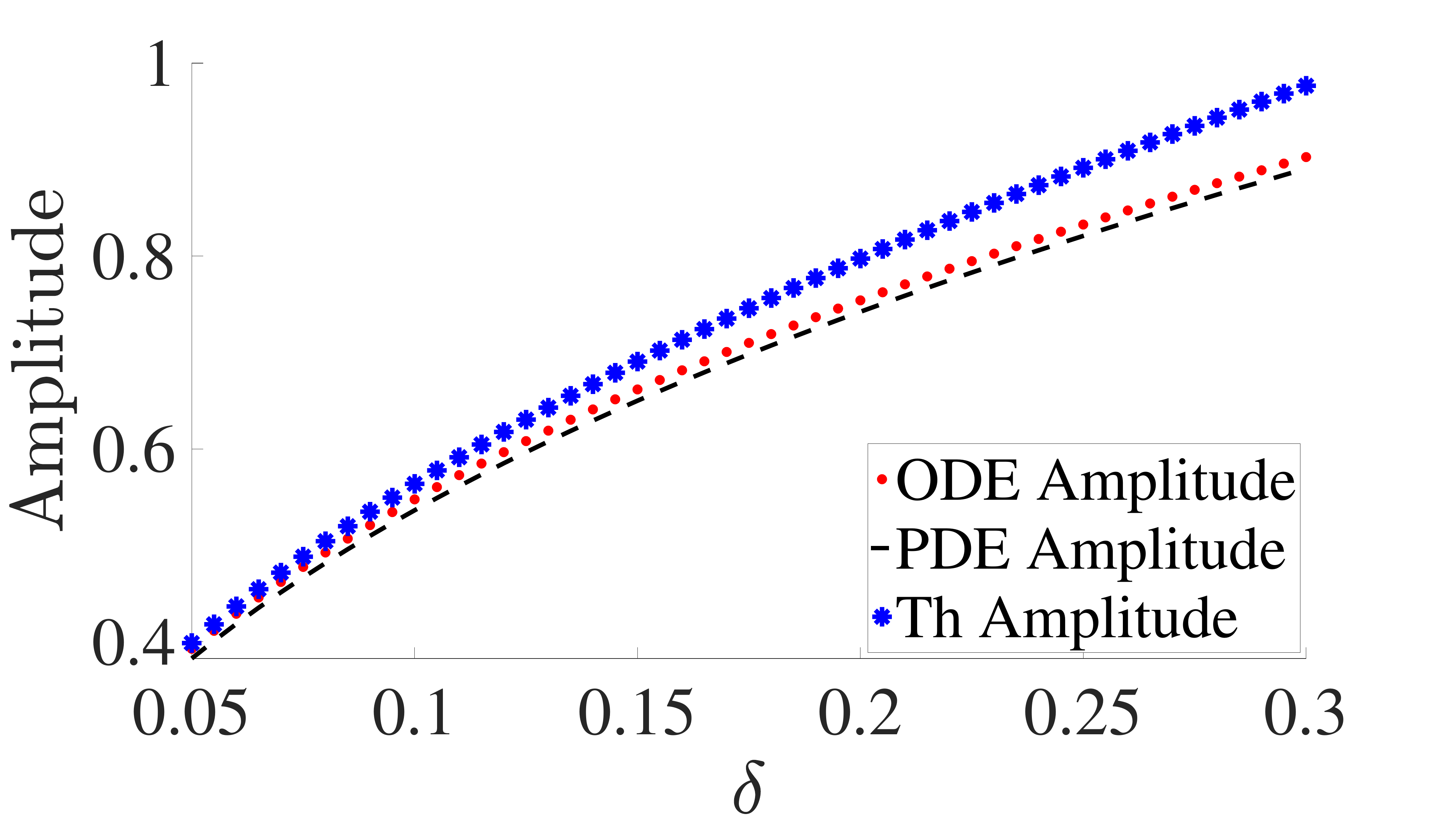}}
\hfil
    \subfloat[]{\includegraphics[width=0.5\linewidth]{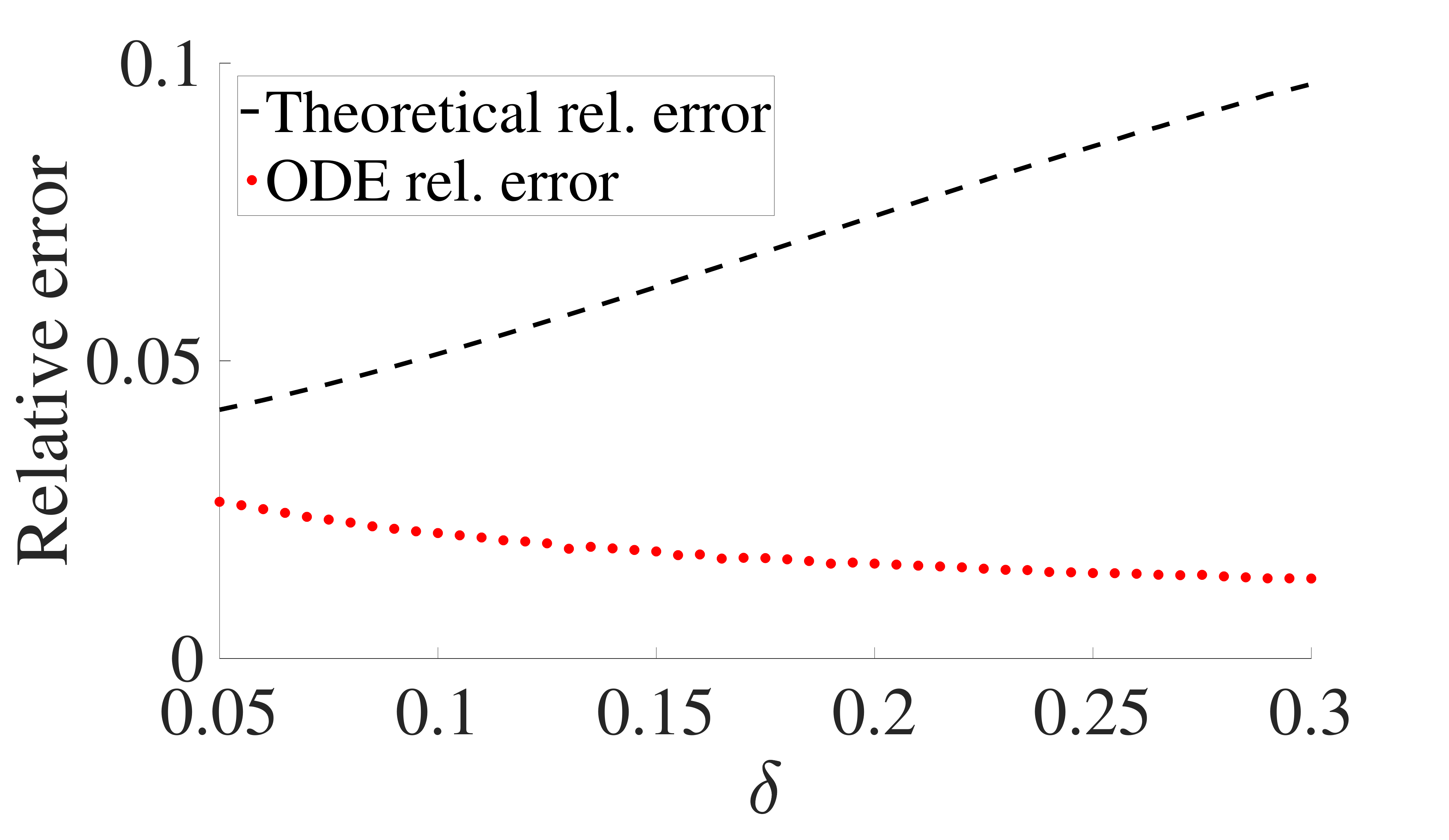}}
    \caption{Two-row model. (a) Displacement amplitude vs. relative distance to instability: the red curve shows the amplitude from the ODE \eqref{eq: ODE p1 q1 th general}, the black curve from the PDE \eqref{eq: 2 layers full PDE} and the theoretical amplitude from \eqref{eq: th amplitude Hopf parameters choice} in blue. (b) Relative error results between the theoretical amplitude and the PDE (in red), and the ODE and the PDE (in black).}
    \label{fig:ampl_ODE_2rows_err_2rows }
\end{figure}

In Figure \ref{fig:ampl_ODE_2rows_err_2rows }(a), we observe the amplitude of $x(t)$ computed with the PDE, the ODE, and the analytical result \eqref{eq: th amplitude Hopf parameters choice}. As expected, the latter loses its prediction capability as we go away from the instability; and subsequently, the relative error between the formula \eqref{eq: th amplitude Hopf parameters choice} and the one measured from the PDE increases, as we can see in Figure \ref{fig:ampl_ODE_2rows_err_2rows }(b). This is due to the contribution of the non linear terms, which play a role in determining the amplitude of oscillations even when the system is close to $\Omega_0$. \\

Note that the two-row model does not only center the equilibrium point for the oscillations around zero, but it also influences important physical quantities such as the amplitude of oscillations. This is emphasized in Figure \ref{fig: amplitude 1rowVS2rows} where the theoretical expansions for both models are plotted (the expansion \eqref{eq: th amplitude Hopf} for the two-row model can be adapted to the one-row model). 
\begin{figure}[hbtp!]
\centering
    \includegraphics[width=0.5\linewidth]{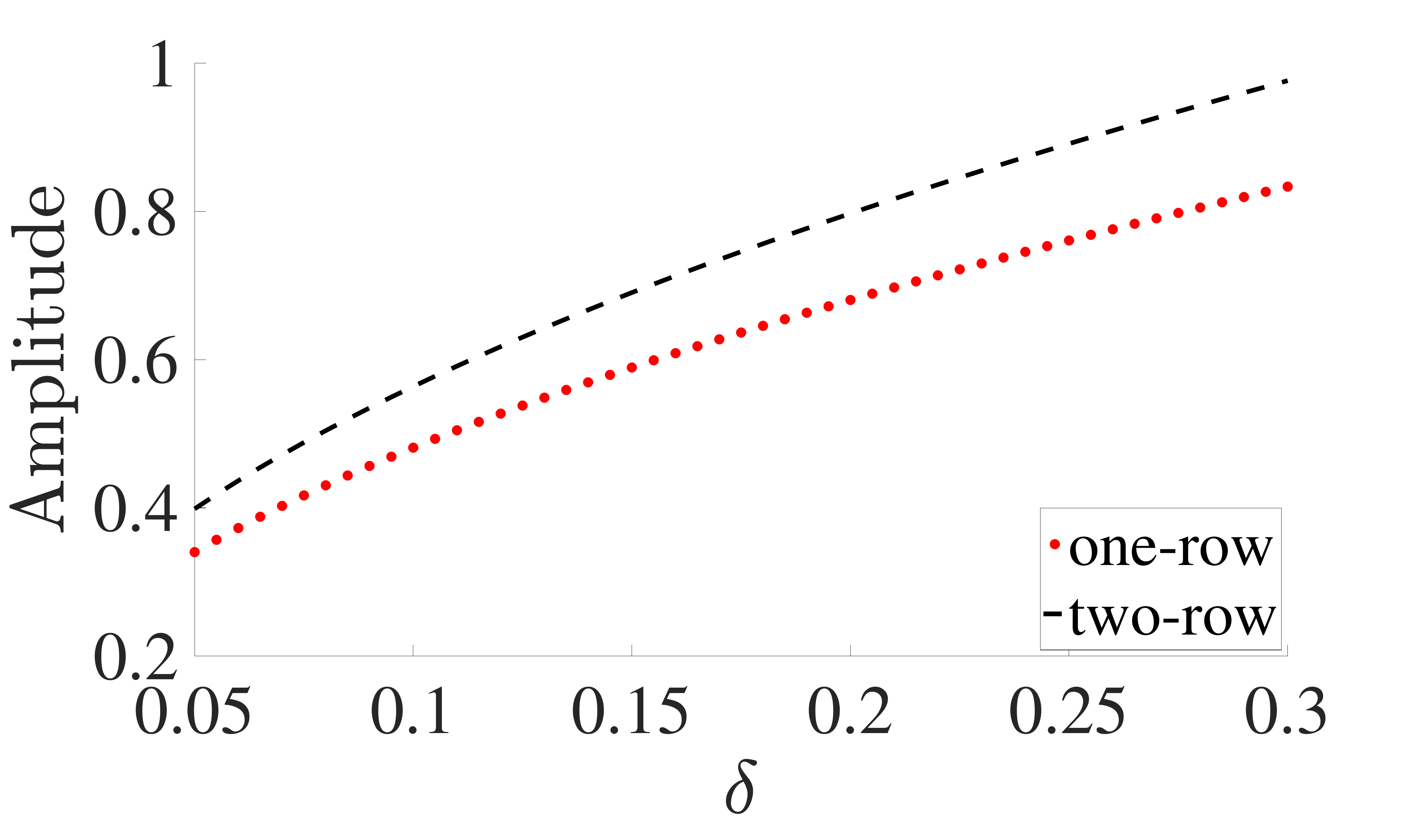}

  \caption{Theoretical amplitude for the two-row model computed with formula \eqref{eq: th amplitude Hopf} in black and its equivalent for the one-row model, in red, against the distance from instability $\delta$.}
    \label{fig: amplitude 1rowVS2rows}
    \end{figure}

\medskip
\subsubsection{\corr{Parameters sensitivity}}

\corr{In this section, we examine the sensitivity of the two-row model ODE \eqref{eq: ODE p1 q1 th general} to the physical parameters $ k $, $ \eta $ and $ \ell $, spring stiffness, viscosity and length of the slice, respectively. To assess their impact on the system, we analyze the oscillation amplitude and frequency by varying one parameter while keeping the others fixed. The results, shown in Figure \ref{fig:SA}, are obtained using MATLAB’s \texttt{ode45} solver at $ \Omega = \Omega_0(1+\delta) $ with $ \delta=0.05 $. 

Our findings indicate that changes in $ \ell $ mostly affect the oscillation amplitude, that goes from 0.3 to 0.8, while the frequency remains nearly constant, decreasing from 7.99 Hz to 7.98 Hz. Conversely, variations in the viscosity coefficient $ \eta $ and the elastic coefficient $ k $ change the frequency, with minimal impact on amplitude. Specifically, as $ k $ increases, the amplitude shifts from 0.54 to 0.56, while the frequency rises from 60 Hz to 100 Hz. In contrast, increasing $ \eta $ reduces the amplitude from 0.56 to 0.54 and lowers the frequency from 120 Hz to 60 Hz.}
\begin{figure}[hbtp!]
\centering
  \subfloat[]{\includegraphics[width=0.5\linewidth]{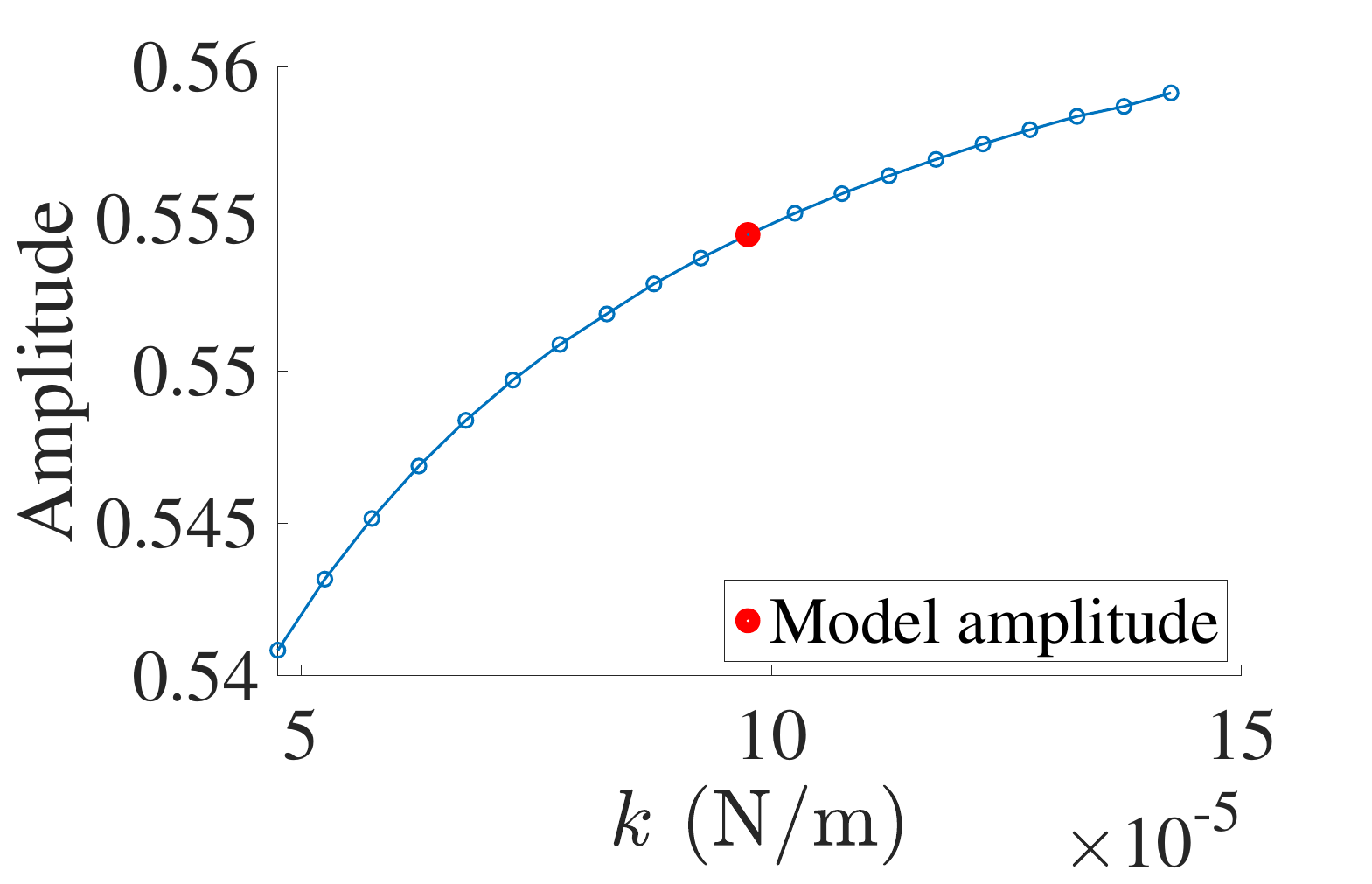}}
\hfil
    \subfloat[]{\includegraphics[width=0.5\linewidth]{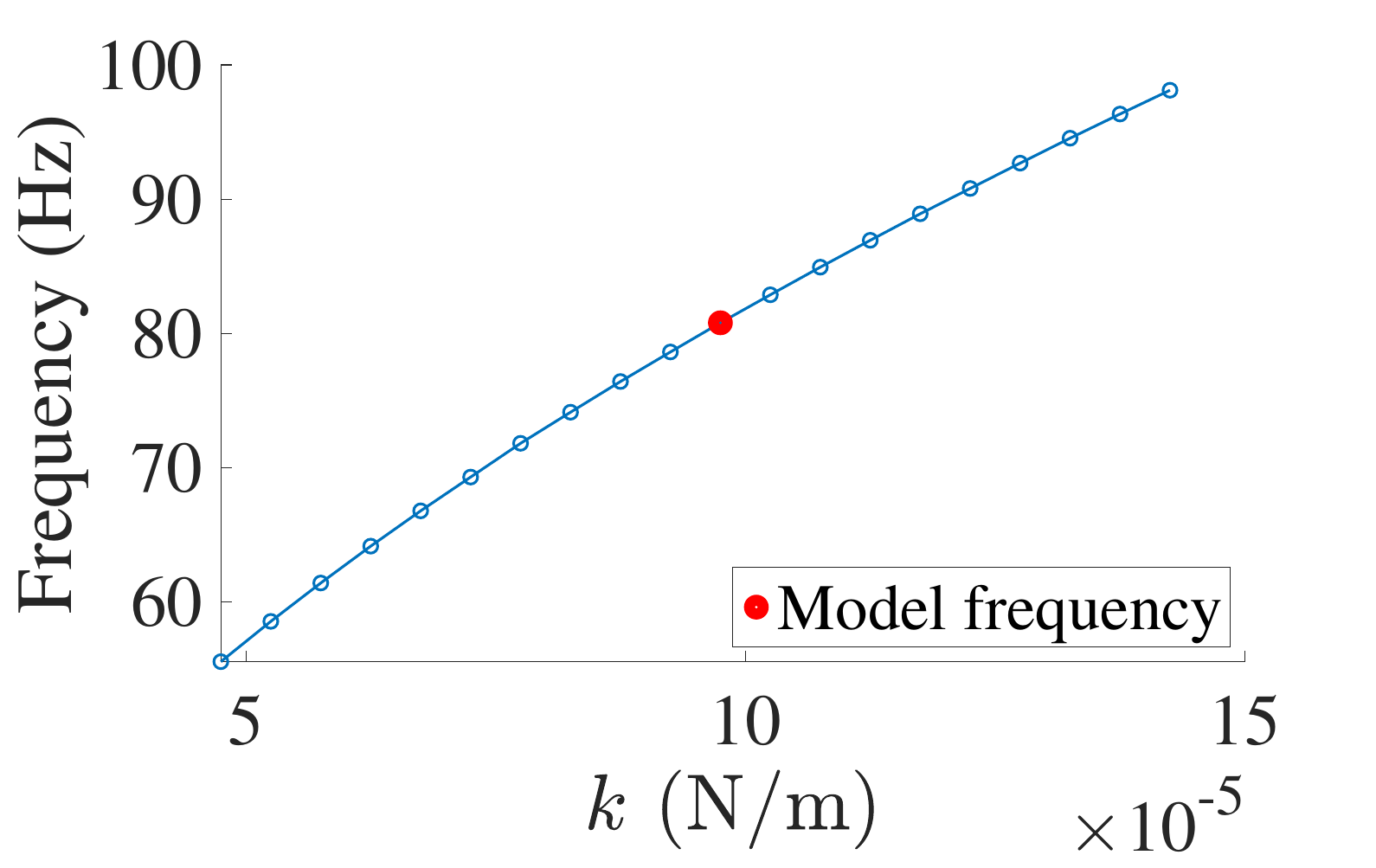}}

\vspace{0.4cm}
\centering
  \subfloat[]{\includegraphics[width=0.5\linewidth]{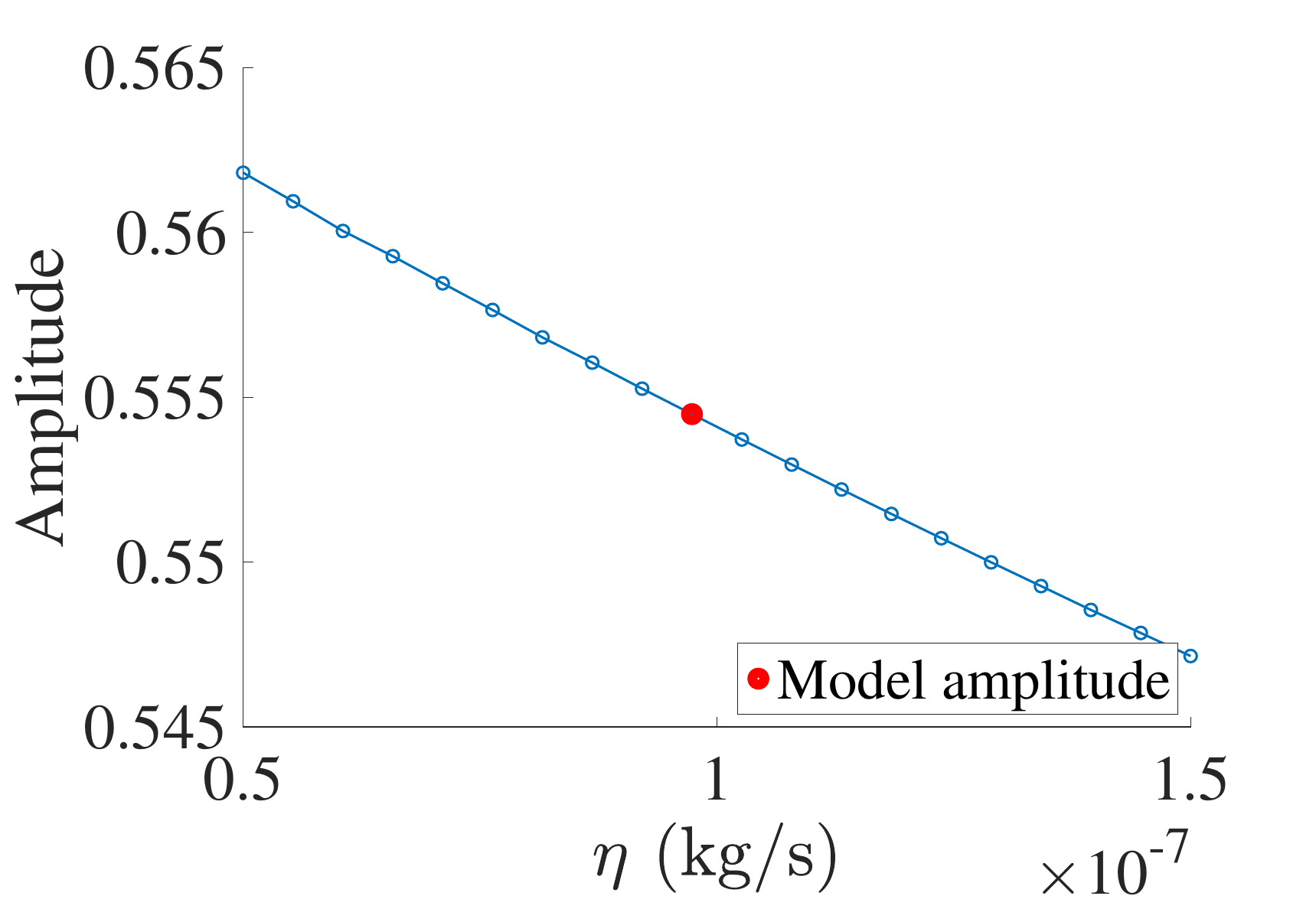}}
\hfil
    \subfloat[]{\includegraphics[width=0.5\linewidth]{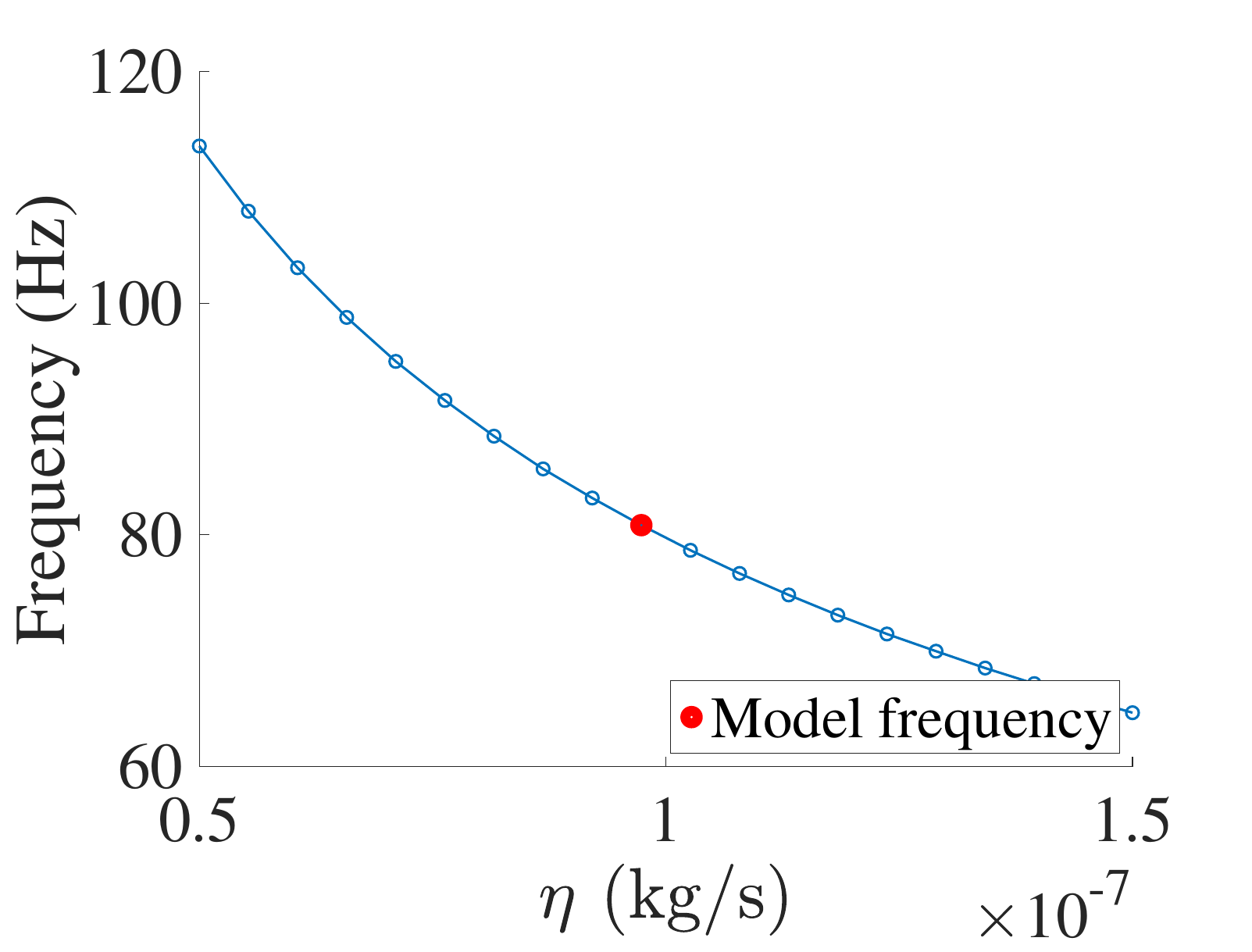}}
    
\vspace{0.4cm}
\centering
  \subfloat[]{\includegraphics[width=0.5\linewidth]{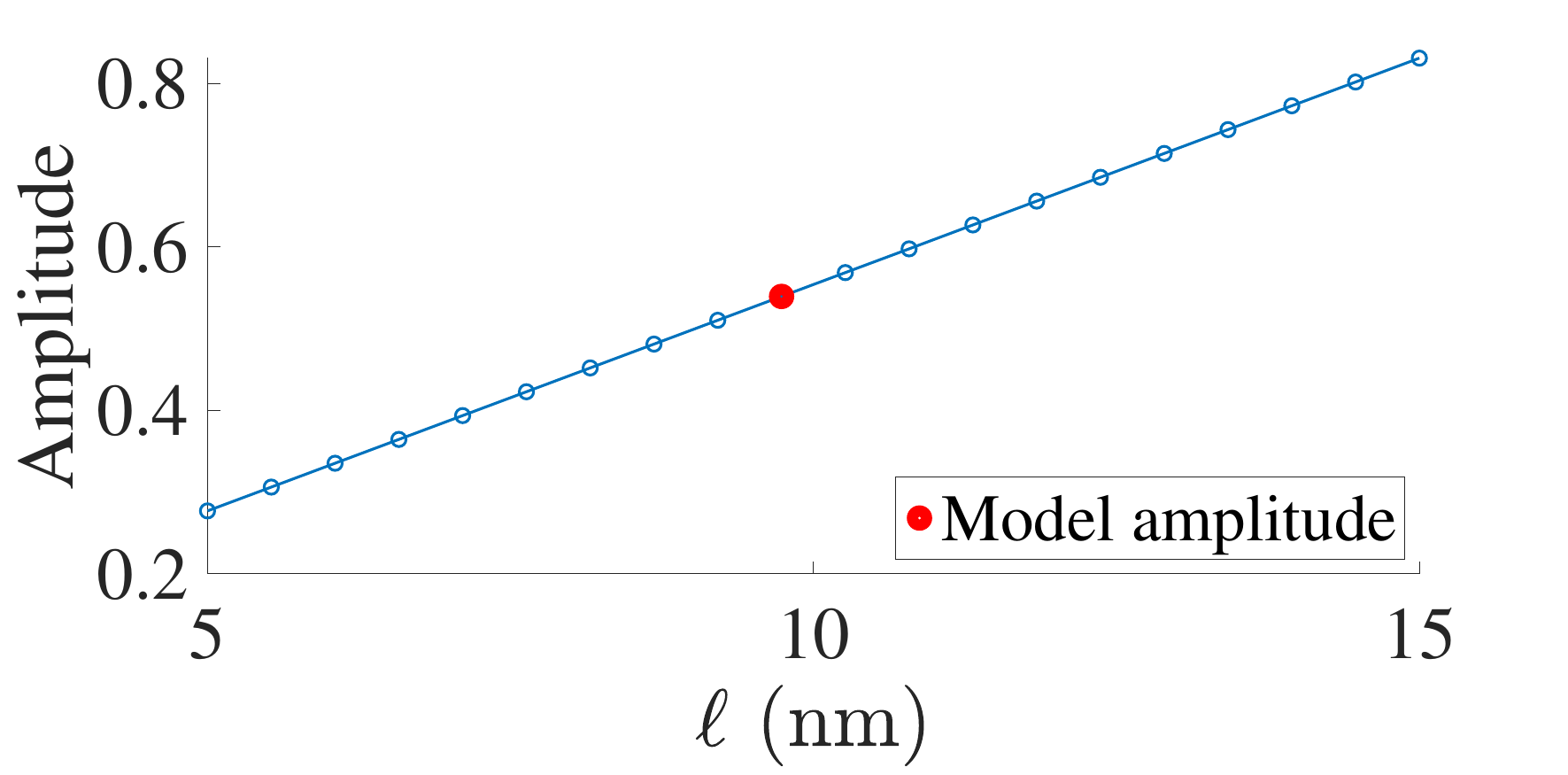}}
\hfil
    \subfloat[]{\includegraphics[width=0.5\linewidth]{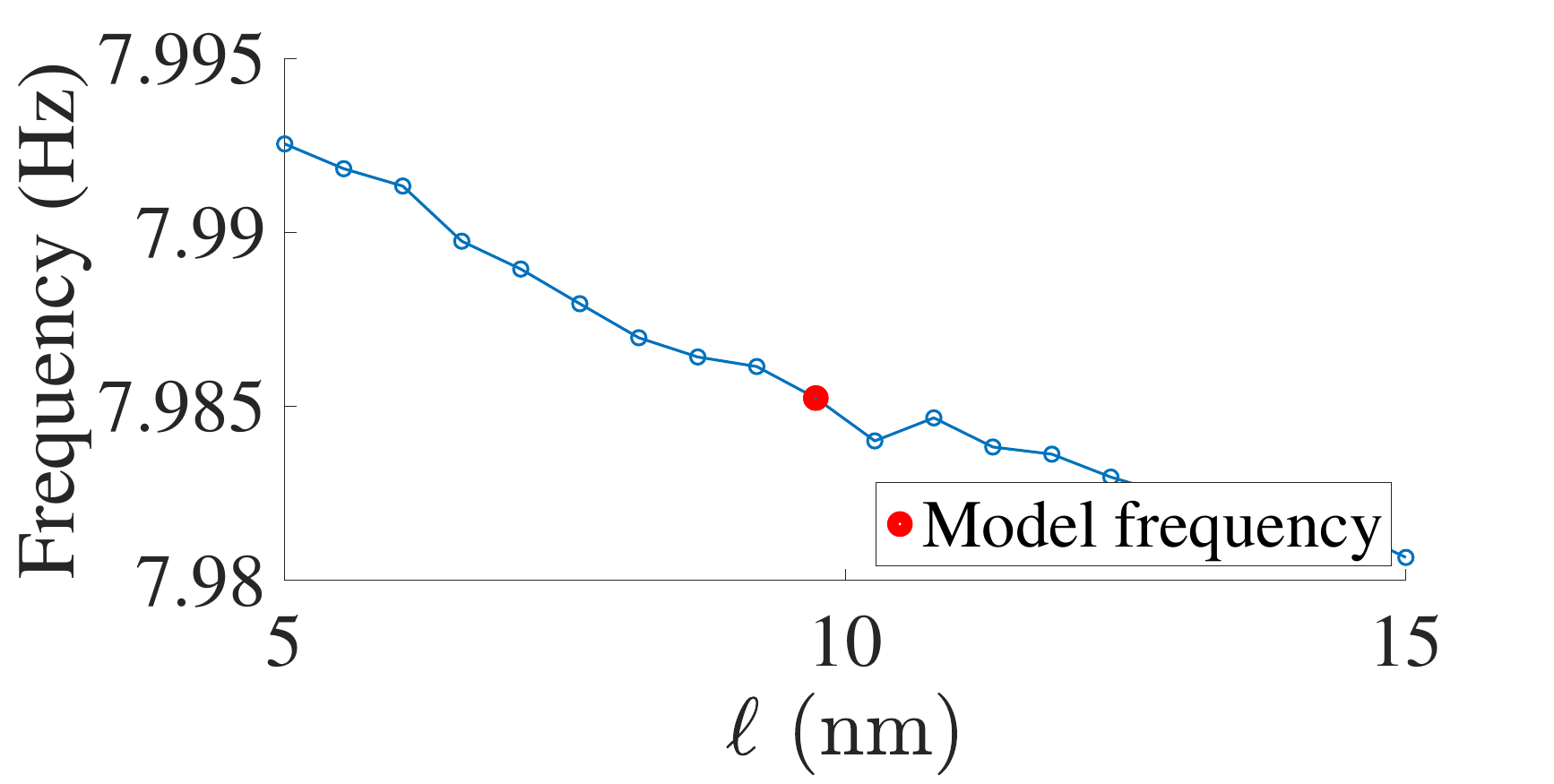}}
  \caption{Amplitude (first column) and frequency of oscillation (second column) of the two-row model ODE \eqref{eq: ODE p1 q1 th general} by varying the physical parameters ($k$, $\eta$, $\ell$). The system was captured at the critical value of $\Omega= \Omega_0(1.05)$. In red $k$, $\eta$, $\ell$ are the model parameter shown in Table 1, they where all decreased and increased of 5\% to perform sensitivity analysis.}
    \label{fig:SA}
\end{figure}

\subsection{$N$-row model: Numerical results}
In this section, the $N$-layer model is tested for $N=8$, to match the $8$ motor rows of an axoneme with a bridge and nine microtubule doublets, as in Figure \ref{fig:simple_axoneme}. We use the same parameter values as for the previous systems, as specified in Table 1.
The initial conditions are such that $Q_i(t=0) = Q_{eq,i}$ for all $i$, and we destabilize the $x_i$s randomly in $[-\bar x, \bar x]$, where $\bar x = 0.01 \, nm$.

Using the same notation as before, 
i.e. $\Omega = (1+\delta) \Omega_0$, we measure the relative distance to the bifurcation point $\Omega_0$ using $\delta$.

\begin{figure}[ht!]
\centering
  \subfloat[]{\includegraphics[width=0.5\linewidth]{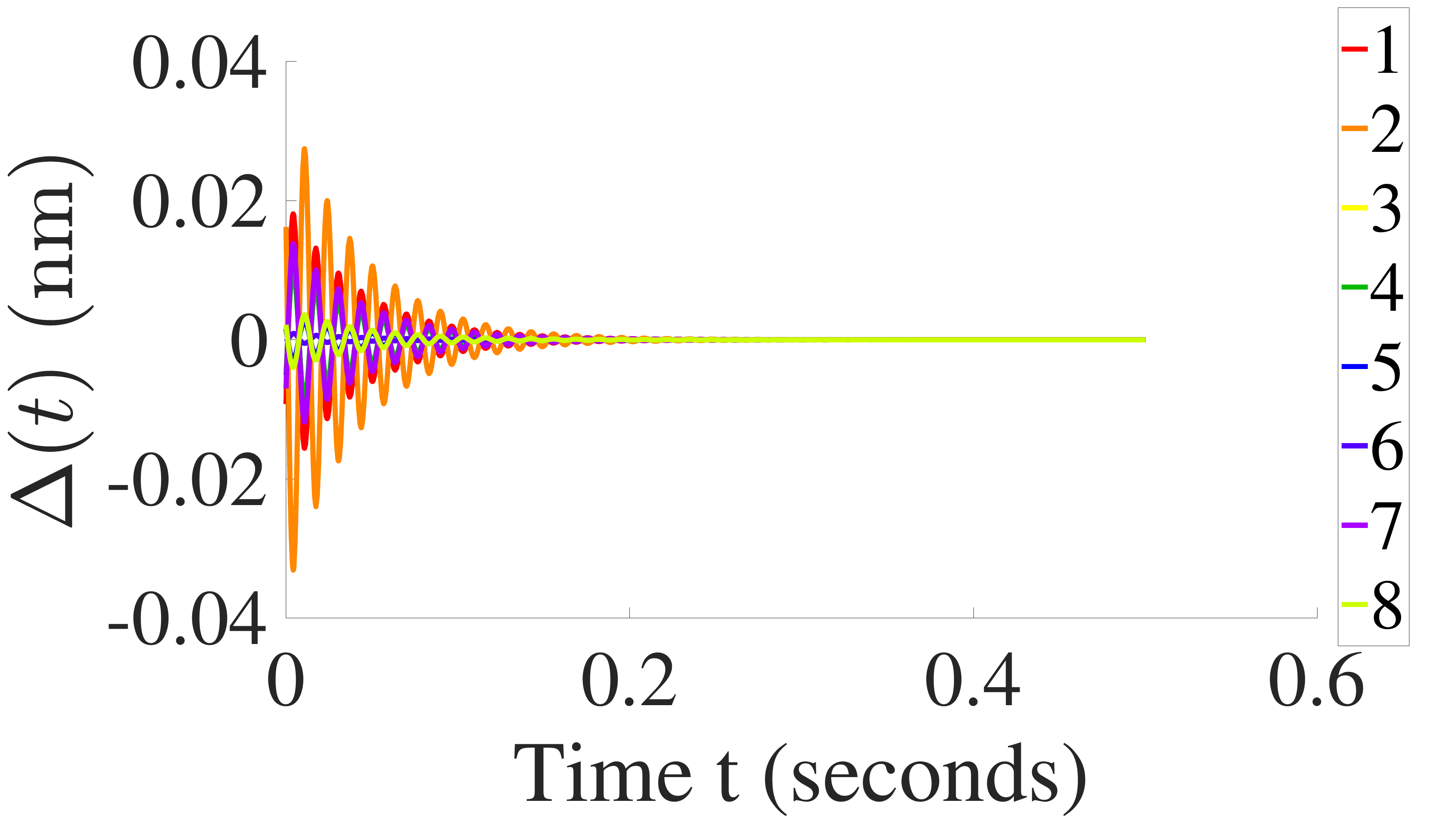}}
\hfil
  \subfloat[]{\includegraphics[width=0.5\linewidth]{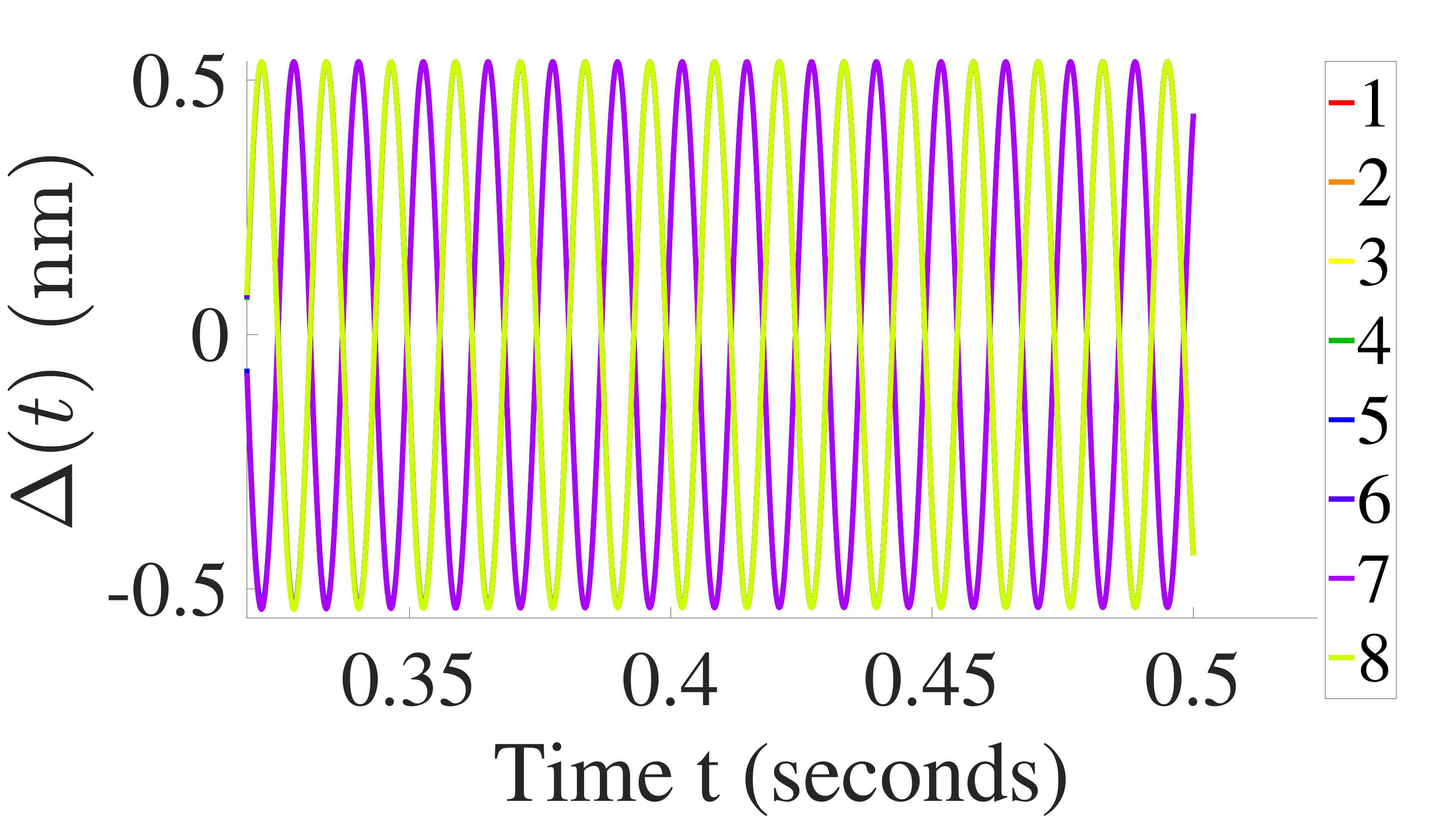}}
\subfloat[]{\includegraphics[width=0.5\linewidth]{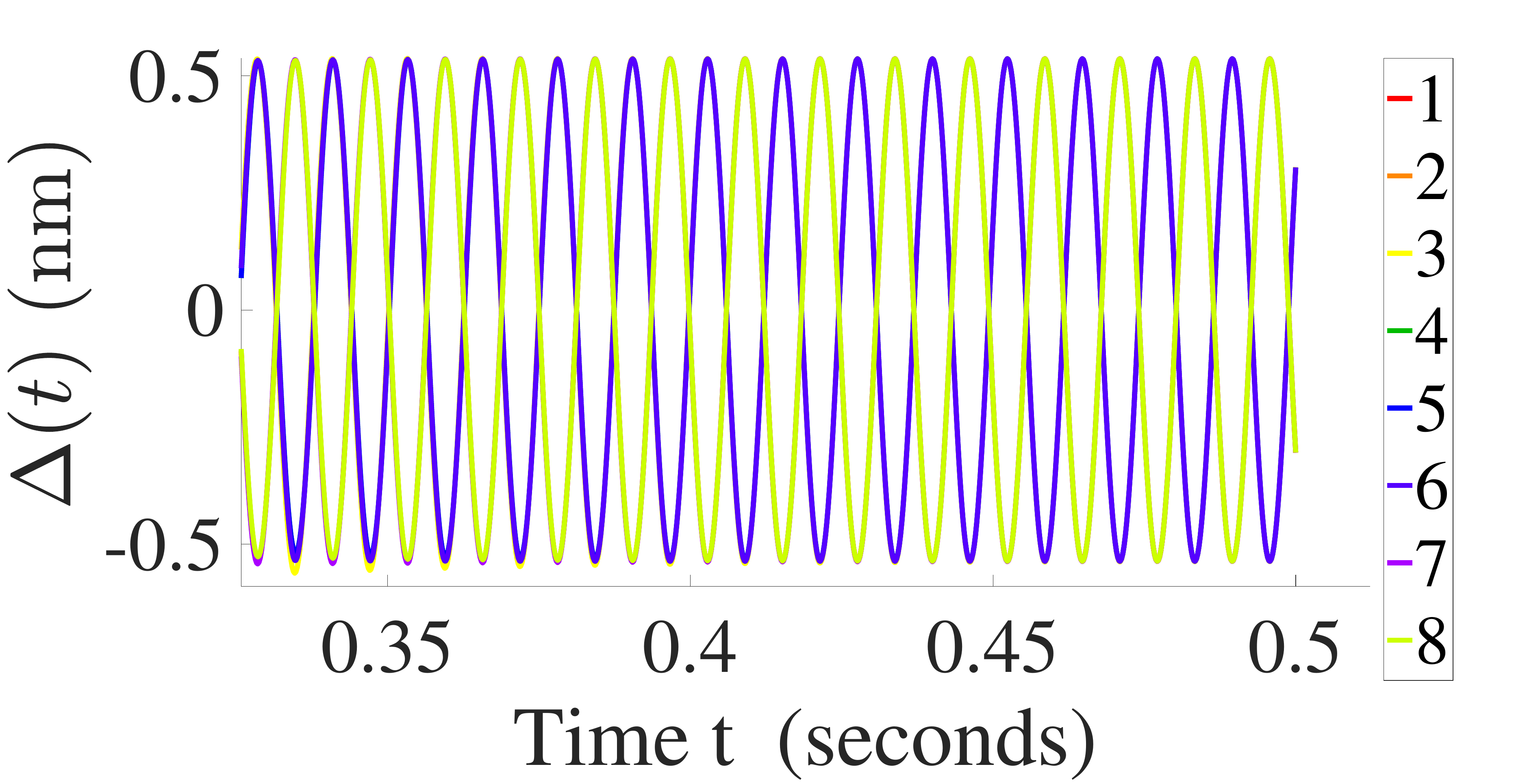}}
	\caption{$N$-row model ($N=8$). Relative tubule displacement (in nm) with respect to time (in seconds) for random initial conditions (a) before bifurcation point, for $\delta = -0.1$ ($\Omega =0.9 \Omega_0$) and (b) and (c) after bifurcation point , for $\delta = 0.1$ ($\Omega = 1.1 \Omega_0$), for two sets of random initial conditions.}
 \label{fig: Nrows}
\end{figure}

Figure \ref{fig: Nrows}(a) shows the absence of oscillations before instability ($\Omega = 0.9\Omega_0$, i.e. $\delta=-0.1$), exactly as expected. Even though the system is put out of equilibrium by the initial conditions, since some tubule shifts are nonzero, they all quickly go back to zero and stop moving.


In Figure \ref{fig: Nrows}(b), we look at the system past the instability point, for $\Omega = 1.1 \Omega_0$ i.e. $\delta=0.1$. Here, the system clearly reaches a steady-state oscillating regime after a short amount of time. All oscillations also remain centered in zero. As expected from the linear study in section \ref{sec: Nbifurcation}, all $\Delta_i$s have the same amplitude and oscillation frequency, but their phase difference varies depending on initial conditions. In the particular case shown in Figure \ref{fig: Nrows}(b), one group of tubule pairs is synchronized with pair $7$, and the other group with pair $8$. In fact, when looking at all layers separately, one can see that odd layers (respectively even layers) oscillate in sync. The other possible outcome when running the simulation with random initial conditions was groups $\{1,\,2,\,7,\,8\}$ and $\{3,\, 4,\,5,\,6\}$ oscillating together, as shown in Figure  \ref{fig: Nrows}(c). 
This influence of the initial conditions is understandable since there is no external force taking the whole axonemal structure and filament into account, resulting in limited coupling between layers at this level. The layers thus tend to keep their original phase difference.
The synchronization shown in Figures \ref{fig: Nrows}(b) and \ref{fig: Nrows}(c) was predicted by Howard et al. \cite{Howard2022}, and proves the importance of having a bridge around which the system alternates between positive and negative displacement.
In the general case, since there is no theory behind the bifurcation of this system of $3N-1$ equations, formally understanding the coupling between phase difference and initial conditions remains an open question. 

A similar problem has recently been numerically studied in Kuramoto oscillators \cite{BAYANI2023}. The Kuramoto model is fairly different from ours but exhibit synchronization of coupled oscillators. 
Some first steps, only modeling individual motors along a single row, have been presented in \cite{Costantini_2024}.

\section{Conclusion and outlook}

In this paper, we propose a model that takes into account the cylindrical structure of an isolated axonome. Compared to the existing literature, this model generalizes previous model by \cite{Julicher1995} that possess only one row of molecular motors to the case of $N$, where $N$ represents the number of rows of motors walking between two microtubule doublets that are arranged in a circular configuration. 

The first result on our model can be appreciated starting from the $N=2$ case. The cylindrical structure answers the problem of the axonemal symmetrization proposed in \cite{camalet2000generic}, previously solved by choosing symmetric transition rates. Instead, in of this paper, general transition rates have been used all along. More in general, even in the $N$-row model, we observed that the equilibrium displacement with null external force was zero, independently from the potentials and the transition rates. This means that the desired symmetry of the axoneme is always preserved: at equilibrium, with no external forces, each microtubule has the same role and there is no initial shifting in the structure. We underline that this is made possible in our model without imposing specific symmetry for the potentials.

The problem appears as a integral partial differential equation for which we provide a mathematical framework to show its well posedness.

In the case where $N=2$, we give the necessary and sufficient conditions for the system to get an Hopf bifurcation depending on the ATP concentration $\Omega$. Using the center manifold theory we have also shown that this bifurcation is supercritical. Physically it means that the molecular motors spontaneously oscillate under varying ATP concentrations around the axoneme.

We analyze the cases of $N=2$ and $N=8$ using both theoretical methods and numerical simulations. For the case $N=8$, under appropriate initial conditions, we note the emergence of two out-of-phase groups among the microtubules, oscillating around the zero position with opposite shifting directions.\\


A natural extension to our work would be to take into account the presence of the central microtubule pair and, therefore, of the radial spokes coming from outer doublets towards the axonemal center. Moreover, this model can be thought of as a building block that, coupled with mechanics of an elastic flagellum, gives the feedback that generates oscillatory patterns following in the footsteps of \cite{camalet2000generic, gadelha2023, oriola2017nonlinear, Walker2020, anello2024beatingeukaryoticflagellahopf} for the planar case and of \cite{sartori2016twist} for the three dimensional case.



\section*{Funding}
This work was supported by a public grant as part of the Investissement d'avenir project, reference ANR-11-LABX-0056-LMH, LabEx LMH.

\corrThm{\appendix
\section{Details on Theorem \ref{th: Hopf bifurcation}}
\subsection{Center manifold computation}\label{appendix: center manifold computations}
From now on we define $a_{10}(\Omega)= a_1(\Omega)/a_0(\Omega)$ for simplicity.

Equation \eqref{eq: N(h)=0 general} is a system of three equations. We focus here on the first one, recovering the coefficients $a_{1j}$ of $h_1=h_1(x,y,\delta\Omega)$ with $j=1,\dots,6$; the same can be done for the second and third equation, solving for the power series of $h_2$ and $h_3$.

The first equation reads
\begin{equation*}
    \begin{split}
         (2 a_{11} y + a_{12}  x + a_{13}  \delta\Omega\, ) &\left(-a_0(\Omega_0)y - \frac{\lambda}{\ell} a_{10}(\Omega_0) y  - \frac{\zeta}{\ell} a_{10}(\Omega_0)x\right.\\
         &\left.-a_0'(\Omega_0)\delta\Omega\,\,  y
        - \displaystyle\frac{\lambda}{\ell} a_{10}'(\Omega_0)\delta\Omega\,\, y -\displaystyle \frac{\zeta}{\ell} a_{10}'(\Omega_0)\delta\Omega\,\, x \right.\\
        &- \left. \displaystyle\frac{{b_1(\Omega_0)}(h_1+h_2-2 h_3)}{2 {a_1(\Omega_0)}}(\zeta  x+\lambda  y) \right) \\
        & + (a_{12} y  + 2 a_{14}  x + a_{15}  \delta\Omega\, )  \left(-\frac{\beta  \ell}{2 \pi } y -\frac{\alpha  \ell}{2 \pi } x\right) \\
        & + a_0(\Omega_0) h_1 +  a_0'(\Omega_0)\delta\Omega\,\,  h_1\\
        &- (\zeta  x+\lambda  y) \left(\displaystyle\frac{ (a_1^2(\Omega_0) (y+h_3)+b_1^2(\Omega_0) (-h_1+y+h_3)}{{a_1(\Omega_0)} b_1(\Omega_0)}\right) = 0.
    \end{split}
\end{equation*}

Since we only need terms $x$, $y$ and $\delta\Omega\,$ up to order two, we get rid of all the third order terms, getting
\begin{equation} \label{eq: first row h1}
    \begin{split}
         (2 a_{11} y + a_{12}  x + a_{13}  \delta\Omega\, ) &\left(-a_0(\Omega_0)y - \frac{\lambda}{\ell} a_{10}(\Omega_0) y  - \frac{\zeta}{\ell} a_{10}(\Omega_0)x \right)\\
        & + (a_{12} y  + 2 a_{14}  x + a_{15}  \delta\Omega\, )  \left(-\frac{\beta  \ell}{2 \pi } y -\frac{\alpha  \ell}{2 \pi } x\right) \\
        & + a_0(\Omega_0) h_1  - y(\zeta  x+\lambda  y)  \left(\displaystyle\frac{ a_1^2(\Omega_0) +b_1^2(\Omega_0)}{{a_1(\Omega_0)} b_1(\Omega_0)}\right) = 0.
    \end{split}
\end{equation}

We then proceed by matching the terms $x$, $y$ and $\delta\Omega\,$ with same order, and solve the resulting system of equations for $a_{1j}$ with $j=1,\dots,6$. Notice we could have kept only the zero order term in the Taylor expansion, since all the higher order terms disappear during this approximation.

From the first two equation we get the same sets of coefficients, namely we obtain for $i=1,2$
\begin{equation*}
\begin{split}
    &a_{i1}= -\frac{\pi \lambda ^2}{\ell(2 \zeta\ell + \pi a_0(\Omega_0))}\frac{a_1^2(\Omega_0)+b_1^2(\Omega_0)}{a_0^2(\Omega_0)b_1(\Omega_0)}\\
    & a_{i2}= \frac{\zeta( \zeta \ell - \pi a_0(\Omega_0))}{(2 \zeta + \pi a_0(\Omega_0))}\frac{a_1^2(\Omega_0)+b_1^2(\Omega_0)}{a_0(\Omega_0)b_1(\Omega_0)}\\
    &a_{i3}= 0, \quad a_{i4}= -\frac{\pi \zeta^2}{\ell(2 \zeta \ell + \pi a_0(\Omega_0))}\frac{a_1^2(\Omega_0)+b_1^2(\Omega_0)}{a_0^2(\Omega_0)b_1(\Omega_0)}\\
    &a_{i5}= 0, \quad a_{i6}= 0.
\end{split}
\end{equation*}
Instead, for $i=3$ we get
\begin{equation*}
\begin{split}
    &a_{31}= -\frac{\pi \lambda ^2}{\ell(2 \zeta\ell + \pi a_0(\Omega_0))}\frac{b_1(\Omega_0)}{a_0^2(\Omega_0)}\\
    & a_{32}= \frac{\zeta( \zeta \ell - \pi a_0(\Omega_0))}{(2 \zeta + \pi a_0(\Omega_0))}\frac{b_1(\Omega_0)}{a_0(\Omega_0)a_1(\Omega_0)}\\
    &a_{33}= 0, \quad a_{34}= -\frac{\pi \zeta ^2}{\ell(2 \zeta\ell + \pi a_0(\Omega_0))}\frac{b_1(\Omega_0)}{a_0^2(\Omega_0)}\\
    &a_{35}= 0, \quad a_{36}= 0.
\end{split}
\end{equation*}

\subsection{Computation for $\tilde \tau$} \label{app: tilde tau}
We now consider system \eqref{eq: Hopf ODE final} at the bifurcation point
\begin{equation} 
     \frac{d}{dt}\tilde{Y} = \begin{pmatrix}
    0 &-\omega(\Omega_0)\\
        \omega(\Omega_0) &0
\end{pmatrix} \tilde{Y} + \begin{pmatrix}
    0\\
    \tilde f^t (x,y)
\end{pmatrix}+O(|\delta\Omega|^3+ |Y|^3).
\end{equation}
where \[ \begin{pmatrix}
    0\\
    \tilde f^t (x,y)
\end{pmatrix}=\mathbb{P}^{-1}f(\mathbf{h}(\mathbb{P}\tilde Y),\mathbb{P}\tilde Y)\] is the transformed non linear term.

The formula for $\tilde \tau$ can be found in \cite{Wiggins} (Chapter 20, Section 2) and is the following
\begin{equation*}
    \tilde{\tau}=-\frac{1}{16 \,\omega(\Omega_0)}\left(\partial^2_{yy}\tilde{f^t}+\partial^2_{xx}\tilde{f^t}\right)\partial^2_{yx}\tilde{f^t}+\frac{1}{16} \left(\partial^3_{yyy}\tilde{f^t}+\partial^3_{xxy}\tilde{f^t}\right),
\end{equation*}
where all the derivatives are evaluated at the bifurcation point, i.e $(0,0,\Omega_0)$.}


\end{document}